\documentclass[reqno,12pt]{amsart}
\usepackage{amsfonts}
\usepackage{bbm}
\usepackage{}
\numberwithin{equation}{section}
\usepackage{color}
\usepackage{rotating}
\usepackage{indentfirst}
\usepackage{color}
\usepackage{amssymb}
\usepackage{mathrsfs}
\usepackage{xy}
\usepackage{rotating}
\usepackage{tikz}

\xyoption{all}
 \def\Hom{\mbox{\rm Hom}}  
 \def\fin{\hfill$\square$}   \def\mod{\mbox{\rm \textbf{mod}}\,}
\def\Thick{\mbox{\rm \textbf{Thick}}\,}
\def\Ker{\mbox{\rm Ker}\,}  
\def\cone{\mbox{\rm cone}}\def\cocone{\mbox{\rm cocone}}
\def\thick{\mbox{\rm thick}\,}
\def\Cone{\mbox{\rm Cone}}\def\Cocone{\mbox{\rm Cocone}}
\def\B{\mathcal {B}}\def\C{\mathcal {C}}
\def\A{\mathcal{A}} 
\def\Id{\mbox{\rm Id}\,} \def\Im{\mbox{\rm Im}\,} \def\add{\mbox{\rm add}\,}

\theoremstyle{plain}
\newtheorem{theorem}{\bf Theorem}[section]
\newtheorem{lemma}[theorem]{\bf Lemma}
\newtheorem{corollary}[theorem]{\bf Corollary}
\newtheorem{proposition}[theorem]{\bf Proposition}

\theoremstyle{definition}
\newtheorem{definition}[theorem]{\bf Definition}
\newtheorem{remark}[theorem]{\bf Remark}
\newtheorem{example}[theorem]{\bf Example}

\newcommand{\bt}{\begin{theorem}}
\newcommand{\et}{\end{theorem}}
\newcommand{\bl}{\begin{lemma}}
\newcommand{\el}{\end{lemma}}
\newcommand{\bd}{\begin{definition}}
\newcommand{\ed}{\end{definition}}
\newcommand{\bc}{\begin{corollary}}
\newcommand{\ec}{\end{corollary}}
\newcommand{\bp}{\begin{proof}}
\newcommand{\ep}{\end{proof}}
\newcommand{\bx}{\begin{example}}
\newcommand{\ex}{\end{example}}
\newcommand{\br}{\begin{remark}}
\newcommand{\er}{\end{remark}}
\newcommand{\be}{\begin{equation}}
\newcommand{\ee}{\end{equation}}
\newcommand{\ba}{\begin{align}}
\newcommand{\ea}{\end{align}}
\newcommand{\bn}{\begin{enumerate}}
\newcommand{\en}{\end{enumerate}}
\newcommand{\bcs}{\begin{cases}}
\newcommand{\ecs}{\end{cases}}
\newcommand{\RNum}[1]{\uppercase\expandafter{\romannumeral #1\relax}}

\makeatletter
\renewcommand{\section}{\@startsection{section}{1}{0mm}
  {-\baselineskip}{0.5\baselineskip}{\bf\leftline}}
\makeatother

\begin{document}
\title[Thick subcategories and silting subcategories in recollement]{Thick subcategories and silting subcategories\\ in recollement }
\author[Y. Mei, L. Wang, J. Wei]{Yuxia Mei, Li Wang, Jiaqun Wei}
\address{School of Mathematics-Physics and Finance, Anhui Polytechnic University, Wuhu 241000, P. R. China.\endgraf}
\email{meiyuxia2010@163.com (Mei), wl04221995@163.com (Wang)}
\address{Department of Mathematics, Zhejiang Normal University, Jinhua 321004, P. R. China.\endgraf}
\email{weijiaqun@njnu.edu.cn (Wei)}


\begin{abstract}
 Let $(\mathcal{A}, \mathcal{B}, \mathcal{C}, i^{*}, i_{\ast}, i^{!},j_!, j^\ast, j_\ast)$ be a recollement of extriangulated categories. We show that there is a bijection between thick subcategories in $\mathcal{C}$ and thick subcategories in $\mathcal{B}$ containing $i_{\ast}\mathcal{A}$. Futhermore, the thick subcategories $\mathcal{V}$ in $\mathcal{B}$ containing $i_{\ast}\mathcal{A}$ can induce a new recollement relative to $\mathcal{A}$ and $j^{\ast}\mathcal{V}$. We also prove that silting subcategories in $\mathcal{A}$ and $\mathcal{C}$ can be glued to get silting subcategories in $\mathcal{B}$ and the converse holds under certain conditions.

\end{abstract}

\maketitle

\section{Introduction}
Recollements of abelian categories were initially introduced by Be{\u\i}linson, Bernstein, and Deligne \cite{BBD} in their construction of the category of perverse sheaves on singular spaces. The concept of recollements in triangulated categories emerged in relation to derived categories of sheaves on topological spaces, where the idea was to "glue" one triangulated category from two others, as discussed in \cite{BBD}. It is well-established that abelian categories and triangulated categories serve as two fundamental frameworks in algebra and geometry.
To simultaneously generalize recollements of both abelian and triangulated categories, the study of recollements in extriangulated categories was undertaken in \cite{WL} by Wang, Zhang, and Wei.

A full subcategory $\mathcal{S}$ of a triangulated category $\mathcal{T}$ is called thick if it is a triangulated subcategory of $\mathcal{T}$ which is in addition closed under taking direct summands. The study of thick subcategories plays a significant role in the exploration of triangulated categories(see \cite{Ric}, \cite{Ver} and so on). Thick subcategories of extriangulated categories were introduced by Adachi and Tsukamoto in  \cite{AT}.
Our first aim is to study the relationship between thick subcategories in the context of a recollement of extriangulated categories.
Additionally, we consider the silting sucbategories which are a special class of
thick subcategories.

Currently, silting theory plays a significant role in the exploration of triangulated categories. In \cite{KY}, it was demonstrated that silting objects correspond to various important structures, including t-structures, co-t-structures, and simple-minded collections. The gluing techniques associated with recollements, as established by Beilinson, Bernstein, and Deligne, have been extensively examined in \cite{BBD} concerning t-structures. The recollements of silting objects \cite{ZW} was investigated in \cite{ZW}. Recently, Wang, Zhang, and Wei \cite{WL} illustrated how to glue two cotorsion pairs from categories $\mathcal{A}$ and $\mathcal{C}$ into a cotorsion pair in $\mathcal{B}$ with respect to a recollement $(\mathcal{A}, \mathcal{B}, \mathcal{C})$ of extriangulated categories. These significant findings prompt the natural inquiry into the gluing of silting subcategories in the context of a recollement of extriangulated categories. Notably, recent work by Adachi and Tsukamoto has established a bijection between bounded hereditary cotorsion pairs and silting subcategories in extriangulated categories \cite{AT}. This bijection will be pivotal in our approach, leading us to investigate whether we can extend this important theory to a more general setting.

Our first aim is to study the relationship between thick subcategories of extriangulated categories in the context of a recollement $(\mathcal{A}, \mathcal{B}, \mathcal{C}, i^{*}, i_{\ast}, i^{!},j_!, j^\ast, j_\ast)$ of extriangulated categories, culminating in the first main theorem of this paper.

\begin{theorem}(Theorem \ref{main-thick-bijection})
There exist mutually inverse bijections
between the set of thick subcategories of $\mathcal{B}$ that contain $i_{\ast}\mathcal{A}$ and the set of thick subcategories of $\mathcal{C}$, where the map given by $\Phi: \mathcal{V} \mapsto j^{\ast}\mathcal{V}$ and the inverse given by
$\Psi: \mathcal{W} \mapsto
\mathcal{V} = \{M \in \mathcal{B} \mid j^{\ast}M \in \mathcal{W}\}$.
\end{theorem}

Our second aim is to investigate the gluing of silting subcategories in relation to a recollement of extriangulated categories, leading to the second main theorem of this paper.

\begin{theorem}(Theorem \ref{main-silt}) Let $(\mathcal{A}, \mathcal{B}, \mathcal{C}, i^{*}, i_{\ast}, i^{!},j_!, j^\ast, j_\ast)$ be a recollement of extriangulated categories as in {\rm (\ref{recolle})}. Suppose that $\mathcal{B}$ has enough projective objects and $i^{!}, j_{!}$ are exact.

$(1)$ Let $\mathcal{M}_{\mathcal{A}}$ and $\mathcal{M}_{\mathcal{C}}$ be silting subcategories in $\mathcal{A}$ and $\mathcal{C}$ respectively.
If the functor $i^{\ast}$ is exact, then
$$\mathcal{M}_{\mathcal{B}}:=\{B\in \mathcal{B}~|~i^{\ast }B\in\mathcal{M}_{\mathcal{A}}^{\vee},j^{\ast}B\in \mathcal{M}_{\mathcal{C}}^{\vee},i^{!}B\in\mathcal{M}_{\mathcal{A}}^{\wedge},j^{\ast}B\in \mathcal{M}_{\mathcal{C}}^{\wedge}\}$$
is a silting subcategory in $\mathcal{B}$. In addition, if $\mathcal{M}_{\mathcal{A}}=\add M_{A}$ and $\mathcal{M}_{\mathcal{C}} = \add M_{C}$, then $M_{\mathcal{B}} = i_{\ast}M_{A} \oplus j_{!}M_{C}$.

$(2)$ Let $\mathcal{M}$ be a silting subcategory in $\mathcal{B}$.
 \begin{itemize}
   \item[(a)]  If $i_{\ast}i^{!}\mathcal{M}^{\vee} \subseteq \mathcal{M}^{\vee}$ and $i_{\ast}i^{\ast}\mathcal{M}^{\vee} \subseteq \mathcal{M}^{\vee}$,
       then $i^{\ast}\mathcal{M}^{\vee} \cap i^{!}\mathcal{M}^{\wedge}$ is a silting subcategory of $\mathcal{A}$.
       In addition, if $\mathcal{M} = \add M$, then $i^{\ast}M$ is a silting object of $\mathcal{A}$.
   \item[(b)] If $j_{\ast}j^{\ast}\mathcal{M}^{\wedge} \subseteq \mathcal{M}^{\wedge}$ or $j_{!}j^{\ast}\mathcal{M}^{\vee} \subseteq \mathcal{M}^{\vee}$,
       then $j^{\ast}\mathcal{M}^{\vee}\cap j^{\ast}\mathcal{M}^{\wedge}$ is a silting subcategory of $\mathcal{C}$.
       In addition, if $\mathcal{M} = \add M$, then $j^{\ast}M$ is a silting object of $\mathcal{C}$.
 \end{itemize}
\end{theorem}


This paper is organized as follows: In the subsequent section, we will discuss properties of extriangulated categories, exact functors, and preliminary results regarding recollements that will be essential for the later sections.
Section 3 focuses on the relationship between thick subcategories of extriangulated categories in the context of a recollement $(\mathcal{A}, \mathcal{B}, \mathcal{C}, i^{*}, i_{\ast}, i^{!},j_!, j^\ast, j_\ast)$.
Section 4 is dedicated to gluing silting subcategories along a recollement of extriangulated categories. More precisely, we investigate how to glue two silting subcategories of $\mathcal{A}$ and $\mathcal{C}$, into a silting subcategory of $\mathcal{B}$ in a manner that is compatible with the bijection between silting subcategories and bounded hereditary cotorsion pairs and the converse holds under certain conditions.
Section 5 is dedicated to giving some applications of the main result in Section 4.

\subsection{Conventions and notation.}
In the context of an additive category $\mathscr{C}$, we assume that its subcategories are both {full} and closed under isomorphisms. For any object $M\in\mathscr{C}$, we denote by $\add M$ the additive closure of $M$, which is the full subcategory of $\mathscr{C}$ consisting of objects that are direct sums of direct summands of $M$.
Let us consider a finite acyclic quiver $Q$. We denote by $S_i$ the one-dimensional simple (left) $kQ$-module associated with the vertex $i$ of $Q$. Furthermore, we denote by $P_i$ the projective cover of $S_i$.

\section{Preliminaries}
Let us recall some notions and properties of extriangulated categories and their recollements from \cite{Na} and \cite{WL}, respectively.

\subsection{Extriangulated categories}
 First we recall from \cite{Na}  the definition of extriangulated categories. We refer the reader to \cite[Section 2]{Na} for more details.

\begin{definition}
(\cite[Definition 2.12]{Na})\label{F}
A triple $(\mathscr{C}, \mathbb{E}, \mathfrak{s})$ is called an {\em extriangulated category} if it satisfies the following conditions.

$\rm(ET1)$  $\mathbb{E}: \mathscr{C}^{op} \times \mathscr{C} \rightarrow Ab$ is an biadditive functor.

$\rm(ET2)$ $\mathfrak{s}$ is an additive realization of $\mathbb{E}$.

$\rm(ET3)$ Let $\delta \in \mathbb{E}(C, A)$ and $\delta' \in \mathbb{E}(C', A')$ be a  pair of $\mathbb{E}$-extensions, realized as
$\mathfrak{s}(\delta)=[A\stackrel{x}{\longrightarrow}B\stackrel{y}{\longrightarrow}C]$, $\mathfrak{s}(\delta')=[A'\stackrel{x'}{\longrightarrow}B'\stackrel{y'}{\longrightarrow}C']$. For any commutative square
$$\xymatrix{
  A \ar[d]_{a} \ar[r]^{x} & B \ar[d]_{b} \ar[r]^{y} & C \\
  A'\ar[r]^{x'} &B'\ar[r]^{y'} & C'}$$
in $\mathscr{C}$, there exists a morphism $(a,c)$: $\delta\rightarrow\delta'$ such that $cy=y'b$.

$\rm(ET3)^{op}$ Dual of $\rm(ET3)$.

$\rm(ET4)$~Let $\delta\in\mathbb{E}(D,A)$ and $\delta'\in\mathbb{E}(F,B)$ be $\mathbb{E}$-extensions realized by
$A\stackrel{f}{\longrightarrow}B\stackrel{f'}{\longrightarrow}D$ and $B\stackrel{g}{\longrightarrow}C\stackrel{g'}{\longrightarrow}F$, respectively.
Then there exists an object $E\in\mathscr{C}$, a commutative diagram
\begin{equation*}\label{2.1}
\xymatrix{
  A \ar@{=}[d]\ar[r]^-{f} &B\ar[d]_-{g} \ar[r]^-{f'} & D\ar[d]^-{d} \\
  A \ar[r]^-{h} & C\ar[d]_-{g'} \ar[r]^-{h'} & E\ar[d]^-{e} \\
   & F\ar@{=}[r] & F   }
\end{equation*}
in $\mathscr{C}$, and an $\mathbb{E}$-extension $\delta''\in \mathbb{E}(E,A)$ realized by $A\stackrel{h}{\longrightarrow}C\stackrel{h'}{\longrightarrow}E$, which satisfy the following:\\
$(\textrm{i})$ $D\stackrel{d}{\longrightarrow}E\stackrel{e}{\longrightarrow}F$ realizes $f'_{\ast}\delta'$,\\
$(\textrm{ii})$ $d^{\ast}\delta''=\delta$,\\
$(\textrm{iii})$ $f_{\ast}\delta''=e^{\ast}\delta'$.\\
$\rm(ET4)^{op}$ Dual of $\rm(ET4)$.
\end{definition}

In this section, we  always assume that $(\mathscr{C}, \mathbb{E}, \mathfrak{s})$ is an extriangulated category. If $\mathfrak{s}(\delta) = [A \stackrel{x}{\longrightarrow} B \stackrel{y}{\longrightarrow} C]$, then the sequence $A \stackrel{x}{\longrightarrow} B \stackrel{y}{\longrightarrow} C$ is called a {\em conflation}, where $x$ is called an {\em inflation} and $y$ is called a {\em deflation}. In this context, the sequence $A \stackrel{x}{\longrightarrow} B \stackrel{y}{\longrightarrow} C \stackrel{\delta}\dashrightarrow$ is called an $\mathbb{E}$-triangle. If necessary, we will express $A = \cocone(y)$ and $C = \cone(x)$. Furthermore, we say that an $\mathbb{E}$-triangle is {\em splitting} if it realizes 0.

An object $P$ in $\mathscr{C}$ is called {\em projective} if, for any conflation $A \stackrel{x}{\longrightarrow} B \stackrel{y}{\longrightarrow} C$ and for any morphism $c \in \mathscr{C}(P, C)$, there exists a morphism $b \in \mathscr{C}(P, B)$ such that $yb = c$. The full subcategory of projective objects in $\mathscr{C}$ is denoted by $\mathcal{P}(\mathscr{C})$.
Dually, we define the {\em injective} objects, and the full subcategory of injective objects in $\mathscr{C}$ is denoted by $\mathcal{I}(\mathscr{C})$. We say that $\mathscr{C}$ {\em has enough projectives} if, for every object $M \in \mathscr{C}$, there exists an $\mathbb{E}$-triangle $A \stackrel{}{\longrightarrow} P \stackrel{}{\longrightarrow} M \stackrel{}\dashrightarrow$ such that $P \in \mathcal{P}(\mathscr{C})$. Dually, we say that $\mathscr{C}$ {\em has enough injectives}.
In particular, if $\mathscr{C}$ is a triangulated category, then $\mathscr{C}$ has enough projectives and injectives, with $\mathcal{P}(\mathscr{C})$ and $\mathcal{I}(\mathscr{C})$ consisting of zero objects.

The higher positive extensions $\mathbb{E}^{k}(-, -)$ in an extriangulated category have recently been defined for $k\geq0$ in \cite{Go}. If $\mathscr{C}$ has enough projective objects or enough injective objects, these $\mathbb{E}^{k}$  are isomorphism to those defined in  \cite{Liu} (cf. \cite[Remark 3.4]{Go}).

\begin{lemma}\text{\rm(\cite[Definition 3.5]{Go})}\label{Long-exact} Let $A\stackrel{x}\longrightarrow B\stackrel{y}\longrightarrow C\stackrel{\delta}\dashrightarrow$ be an $\mathbb{E}$-triangle in $\mathscr{C}$. We have long exact sequences
\begin{equation*}
\cdots \longrightarrow \mathbb{E}^{k}(-, A)\stackrel{} \longrightarrow \mathbb{E}^{k}(-, B)\stackrel{} \longrightarrow \mathbb{E}^{k}(-, C)\stackrel{} \longrightarrow \mathbb{E}^{k+1}(-, A)\stackrel{} \longrightarrow \mathbb{E}^{k+1}(-, B) \longrightarrow \cdots
\end{equation*}
and
\begin{equation*}
\cdots \longrightarrow \mathbb{E}^{k}(C, -)\stackrel{} \longrightarrow \mathbb{E}^{k}(B, -)\stackrel{} \longrightarrow \mathbb{E}^{k}(A, -)\stackrel{} \longrightarrow \mathbb{E}^{k+1}(C, -)\stackrel{} \longrightarrow \mathbb{E}^{k+1}(B,-) \longrightarrow \cdots
\end{equation*}
\end{lemma}

\subsection{Recollements of extriangulated categories}
%
%

A morphism $f$ in $\mathscr{C}$ is called {\em compatible} (cf. \cite[Definition 2.8]{WL}), if ``$f$ is both an inflation and a deflation" implies that $f$ is an isomorphism. For instance, each morphism  in an exact category is compatible. When $\mathscr{C}$ is a  triangulated category, the compatible morphisms are precisely the isomorphisms.

A finite sequence $$X_{n}\stackrel{d_{n}}{\longrightarrow}X_{n-1}\stackrel{d_{n-1}}{\longrightarrow}\cdots \stackrel{d_{2}}{\longrightarrow} X_{1}\stackrel{d_{1}}{\longrightarrow}X_{0}$$ in $\mathscr{C}$ is said to be an {\em $\mathbb{E}$-triangle sequence}, if there exist $\mathbb{E}$-triangles $X_{n}\stackrel{d_{n}}{\longrightarrow}X_{n-1}\stackrel{f_{n-1}}{\longrightarrow}K_{n-1}\stackrel{}\dashrightarrow$,
$K_{i+1}\stackrel{g_{i+1}}{\longrightarrow}X_{i}\stackrel{f_{i}}{\longrightarrow}K_{i}\stackrel{}\dashrightarrow,~~1<i<n-1$,
and $K_{2}\stackrel{g_{2}}{\longrightarrow}X_{1}\stackrel{d_1}{\longrightarrow}X_0\stackrel{}\dashrightarrow$ such that $d_i=g_if_i$ for any $1<i<n$.

\begin{definition}\label{right}  \text{\rm (\cite[Definition 2.9]{WL})}
A sequence $A\stackrel{f}{\longrightarrow}B\stackrel{g}{\longrightarrow}C$ in $\mathscr{C}$ is said to be {\em right exact} if there exists an $\mathbb{E}$-triangle $K\stackrel{h_{2}}{\longrightarrow}B\stackrel{g}{\longrightarrow}C\stackrel{}\dashrightarrow$ and a compatible deflation $h_{1}:A\rightarrow K$ such that $f=h_2h_1$. Dually, one can also define the {\em left exact} sequences.

A $4$-term $\mathbb{E}$-triangle sequence $A{\stackrel{f}\longrightarrow}B\stackrel{g}{\longrightarrow}C\stackrel{h}{\longrightarrow}D$ is called {\em right exact} (resp. {\em left exact}) if there exist $\mathbb{E}$-triangles $A\stackrel{f}{\longrightarrow}B\stackrel{g_1}{\longrightarrow}K\stackrel{}\dashrightarrow$
and $K\stackrel{g_{2}}{\longrightarrow}C\stackrel{h}{\longrightarrow}D\stackrel{}\dashrightarrow$ such that $g=g_2g_1$ and $g_1$ (resp. $g_2$) is compatible.
\end{definition}

%
%
%

\begin{definition}\label{right exact} \text{\rm (\cite[Definition 2.12]{WL})}  Let $(\mathcal{A},\mathbb{E}_{\mathcal{A}},\mathfrak{s}_{\mathcal{A}})$ and $(\mathcal{B},\mathbb{E}_{\mathcal{B}},\mathfrak{s}_{\mathcal{B}})$ be extriangulated categories. An additive covariant functor $F:\mathcal{A}\rightarrow \mathcal{B}$ is called a {\em right exact functor} if it satisfies the following conditions
\begin{itemize}
   \item [(1)] If $f$ is a compatible morphism in $\A$, then $Ff$ is compatible in $\B$.
   \item [(2)] If $A\stackrel{a}{\longrightarrow}B\stackrel{b}{\longrightarrow}C$ is right exact in $\mathcal{A}$, then $FA\stackrel{Fa}{\longrightarrow}FB\stackrel{Fb}{\longrightarrow}FC$ is right exact in $\mathcal{B}$. (Then for any $\mathbb{E}_{\mathcal{A}}$-triangle $A\stackrel{f}{\longrightarrow}B\stackrel{g}{\longrightarrow}C\stackrel{\delta}\dashrightarrow$, there exists an $\mathbb{E}_{\mathcal{B}}$-triangle $A'\stackrel{x}{\longrightarrow}FB\stackrel{Fg}{\longrightarrow}FC\stackrel{}\dashrightarrow$ such that $Ff=xy$ and $y: FA\rightarrow A'$ is a deflation and compatible. Moreover, $A'$ is uniquely determined up to isomorphism.)

  \item [(3)] There exists a natural transformation $$\eta=\{\eta_{(C,A)}:\mathbb{E}_{\mathcal{A}}(C,A)\longrightarrow\mathbb{E}_{\mathcal{B}}(F^{op}C,A')\}_{(C,A)\in{\A}^{\rm op}\times\A}$$ such that $\mathfrak{s}_{\mathcal{B}}(\eta_{(C,A)}(\delta))=[A'\stackrel{x}{\longrightarrow}FB\stackrel{Fg}{\longrightarrow}FC]$.

 \end{itemize}

\end{definition}

\begin{definition}\label{exact functor} \text{\rm (\cite[Definition 2.13]{WL})}
Let $(\mathcal{A},\mathbb{E}_{\mathcal{A}},\mathfrak{s}_{\mathcal{A}})$ and $(\mathcal{B},\mathbb{E}_{\mathcal{B}},\mathfrak{s}_{\mathcal{B}})$ be two extriangulated categories.  We say an additive covariant functor $F:\mathcal{A}\rightarrow \mathcal{B}$ is an {\em exact functor} if the following conditions hold.
\begin{itemize}
  \item [(1)] If $f$ is a compatible morphism in $\A$, then $Ff$ is compatible in $\B$.
  \item [(2)] There exists a natural transformation $$\eta=\{\eta_{(C,A)}\}_{(C,A)\in{\A}^{\rm op}\times\A}:\mathbb{E}_{\mathcal{A}}(-,-)\Rightarrow\mathbb{E}_{\mathcal{B}}(F^{\rm op}-,F-).$$
  \item [(3)] If $\mathfrak{s}_{\mathcal{A}}(\delta)=[A\stackrel{x}{\longrightarrow}B\stackrel{y}{\longrightarrow}C]$, then $\mathfrak{s}_{\mathcal{B}}(\eta_{(C,A)}(\delta))=[F(A)\stackrel{F(x)}{\longrightarrow}F(B)\stackrel{F(y)}{\longrightarrow}F(C)]$.

  \end{itemize}
\end{definition}

\begin{lemma}\label{adjoint}  \text{\rm (\cite[Definition 2.16]{WL})} Let $\mathcal{A}$ and $\mathcal{B}$ be two categories, and let $F:\mathcal{A}\rightarrow \mathcal{B}$ be a functor which admits a right adjoint functor $G$. Let $\eta:\Id_{\mathcal{A}}\Rightarrow GF$ be the unit and $\epsilon:FG\Rightarrow\Id_{\mathcal{B}}$ be the counit.

$(1)$ $\Id_{FX}=\epsilon_{FX}F{(\eta_{X})}$ for any $X\in\mathcal{A}$.

$(2)$ $\Id_{GY}=G(\epsilon_{Y}){\eta_{GY}}$ for any $Y\in\mathcal{B}$.
\end{lemma}

 We always assume that any extrianglated category satisfies the (WIC) condition  (see \cite[Condition 5.8]{Na}).

\begin{definition}\label{recollement} \text{\rm (\cite[Definition 3.1]{WL})}
Let $\mathcal{A}$, $\mathcal{B}$ and $\mathcal{C}$ be three extriangulated categories. A \emph{recollement} of $\mathcal{B}$ relative to
$\mathcal{A}$ and $\mathcal{C}$, denoted by $(\mathcal{A}, \mathcal{B}, \mathcal{C}, i^{*}, i_{\ast}, i^{!},j_!, j^\ast, j_\ast)$, is a diagram
\begin{equation}\label{recolle}
  \xymatrix{\mathcal{A}\ar[rr]|{i_{*}}&&\ar@/_1pc/[ll]|{i^{*}}\ar@/^1pc/[ll]|{i^{!}}\mathcal{B}
\ar[rr]|{j^{\ast}}&&\ar@/_1pc/[ll]|{j_{!}}\ar@/^1pc/[ll]|{j_{\ast}}\mathcal{C}}
\end{equation}
given by two exact functors $i_{*},j^{\ast}$, two right exact functors $i^{\ast}$, $j_!$ and two left exact functors $i^{!}$, $j_\ast$, which satisfies the following conditions:
\begin{itemize}
  \item [(R1)] $(i^{*}, i_{\ast}, i^{!})$ and $(j_!, j^\ast, j_\ast)$ are adjoint triples.
  \item [(R2)] $\Im i_{\ast}=\Ker j^{\ast}$.
  \item [(R3)] $i_\ast$, $j_!$ and $j_\ast$ are fully faithful.
  \item [(R4)] For each $X\in\mathcal{B}$, there exists a left exact $\mathbb{E}_{\B}$-triangle sequence
  \begin{equation}\label{first}
  \xymatrix{i_\ast i^! X\ar[r]^-{\theta_X}&X\ar[r]^-{\vartheta_X}&j_\ast j^\ast X\ar[r]&i_\ast A}
   \end{equation}
  with $A\in \mathcal{A}$, where $\theta_X$ and  $\vartheta_X$ are given by the adjunction morphisms.
  \item [(R5)] For each $X\in\mathcal{B}$, there exists a right exact $\mathbb{E}_{\B}$-triangle sequence
  \begin{equation}\label{second}
  \xymatrix{i_\ast\ar[r] A' &j_! j^\ast X\ar[r]^-{\upsilon_X}&X\ar[r]^-{\nu_X}&i_\ast i^\ast X&}
   \end{equation}
 with $A'\in \mathcal{A}$, where $\upsilon_X$ and $\nu_X$ are given by the adjunction morphisms.
\end{itemize}
\end{definition}

Now, we collect some properties of recollement of extriangulated categories, which will be used in the sequel.

\begin{lemma}\label{CY} \text{\rm (\cite[Lemma 3.3]{WL})} Let $(\mathcal{A}, \mathcal{B}, \mathcal{C}, i^{*}, i_{\ast}, i^{!},j_!, j^\ast, j_\ast)$ be a recollement of extriangulated categories as in (\ref{recolle}).

$(1)$ All the natural transformations
$$i^{\ast}i_{\ast}\Rightarrow\Id_{\A},~\Id_{\A}\Rightarrow i^{!}i_{\ast},~\Id_{\C}\Rightarrow j^{\ast}j_{!},~j^{\ast}j_{\ast}\Rightarrow\Id_{\C}$$
are natural isomorphisms.

$(2)$ $i^{\ast}j_!=0$ and $i^{!}j_\ast=0$.

$(3)$ $i^{\ast}$ preserves projective objects and $i^{!}$ preserves injective objects.

$(3')$ $j_{!}$ preserves projective objects and $j_{\ast}$ preserves injective objects.

$(4)$ If $i^{!}$ (resp. $j_{\ast}$) is  exact, then $i_{\ast}$ (resp. $j^{\ast}$) preserves projective objects.

$(4')$ If $i^{\ast}$ (resp. $j_{!}$) is  exact, then $i_{\ast}$ (resp. $j^{\ast}$) preserves injective objects.

$(5)$ If $\mathcal{B}$ has enough projectives, then $\mathcal{A}$ has enough projectives and $\mathcal{P}(\mathcal{A})=\add(i^{\ast}(\mathcal{P}(\mathcal{B})))$; if $\mathcal{B}$ has enough injectives, then $\mathcal{A}$ has enough injectives and $\mathcal{I}(\mathcal{A})=\add(i^{!}(\mathcal{I}(\mathcal{B})))$.

$(6)$  If $\mathcal{B}$ has enough projectives and $j_{\ast}$ is exact, then $\mathcal{C}$ has enough projectives and $\mathcal{P}(\mathcal{C})=\add(j^{\ast}(\mathcal{P}(\mathcal{B})))$; if $\mathcal{B}$ has enough injectives and $j_{!}$ is exact, then $\mathcal{C}$ has enough injectives and $\mathcal{I}(\mathcal{C})=\add(j^{\ast}(\mathcal{I}(\mathcal{B})))$.

$(7)$ If $\mathcal{B}$ has enough projectives and $i^{!}$ is  exact, then $\mathbb{E}_{\mathcal{B}}(i_{\ast}X,Y)\cong\mathbb{E}_{\mathcal{A}}(X,i^{!}Y)$ for any $X\in\mathcal{A}$ and $Y\in\mathcal{B}$.

$(7')$  If $\mathcal{C}$ has enough projectives and $j_{!}$ is  exact, then $\mathbb{E}_{\mathcal{B}}(j_{!}Z,Y)\cong\mathbb{E}_{\mathcal{C}}(Z,j^{\ast}Y)$ for any $Y\in\mathcal{B}$ and $Z\in\mathcal{C}$.

$(8)$  If $i^{\ast}$ is exact, then $j_{!}$ is  exact.

$(8')$ If $i^{!}$ is exact, then $j_{\ast}$ is exact.
\end{lemma}

%

\section{Recollements of thick subcategories}


First, we recall the definition of thick subcategories in an extriangulated category from \cite{AT}.
Assume that $\mathcal{D}$ and $\mathcal{D}'$ are two subcategories of the extriangulated category $\mathscr{C}$. Let $\mathcal{D}\ast\mathcal{D}'$ denote the subcategory of $\mathscr{C}$ consisting of $X\in\mathcal{C}$ that admit an $\mathbb{E}$-triangle $D\stackrel{}{\longrightarrow}X\stackrel{}{\longrightarrow}D'\stackrel{}\dashrightarrow$ with $D\in\mathcal{D}$ and $D'\in\mathcal{D}'$.
We say that $\mathcal{D}$ is {\em closed under extensions} if $\mathcal{D}\ast\mathcal{D} \subseteq \mathcal{D}$. Let $\text{Cone}(\mathcal{D},\mathcal{D}')$ denote the subcategory of $\mathscr{C}$ consisting of $X\in\mathcal{C}$ that admit an $\mathbb{E}$-triangle
$D\stackrel{}{\longrightarrow} D'\stackrel{}{\longrightarrow}X\stackrel{}\dashrightarrow$ with $D\in\mathcal{D}$ and $D'\in\mathcal{D}'$.
We say that $\mathcal{D}$ is {\em closed under cones} if $\text{Cone}(\mathcal{D},\mathcal{D}) \subseteq \mathcal{D}$.
Let $\text{Cocone}(\mathcal{D},\mathcal{D}')$ denote the subcategory of $\mathscr{C}$ consisting of $X\in\mathcal{C}$ that admit an $\mathbb{E}$-triangle
$X\stackrel{}{\longrightarrow} D\stackrel{}{\longrightarrow}D'\stackrel{}\dashrightarrow$ with $D\in\mathcal{D}$ and $D'\in\mathcal{D}'$.
We say that $\mathcal{D}$ is {\em closed under cocones} if
$\text{Cocone}(\mathcal{D},\mathcal{D}) \subseteq \mathcal{D}$.

\begin{definition}\label{thick} \text{\rm (\cite[Definition 2.4]{AT})}
Assume that $\mathcal{D}$ is a subcategory of the extriangulated category $\mathscr{C}$.
We call $\mathcal{D}$ a {\em thick subcategory} of  $\mathcal{C}$ if it is closed under extensions, cones, cocones and direct summands. Let $\thick\mathcal{D}$ denote the smallest thick subcategory containing $\mathcal{D}$.

\end{definition}

In what follows, we always assume that $(\mathcal{A}, \mathcal{B}, \mathcal{C}, i^{*}, i_{\ast}, i^{!},j_!, j^\ast, j_\ast)$  is a recollement of extriangulated categories as in (\ref{recolle}). We start with the following observation.

\begin{lemma}\label{lemma 31}
Let $\mathcal{V}$ be a thick subcategory of $\mathcal{B}$ such that $i_{\ast}\mathcal{A} \subseteq \mathcal{V}$. Then the following holds.
\begin{itemize}
\item[(1)] $i_{\ast}i^{\ast}\mathcal{V} \subseteq \mathcal{V}$ and $i_{\ast}i^{!}\mathcal{V} \subseteq \mathcal{V}$.
\item[(2)] $j_{\ast}j^{\ast}\mathcal{V} \subseteq \mathcal{V}$ and $j_{!}j^{\ast}\mathcal{V} \subseteq \mathcal{V}$.
\item[(3)] If $j^{\ast}V \in j^{\ast}\mathcal{V}$ for some $V \in \mathcal{B}$, then $V \in \mathcal{V}$.
\end{itemize}
\end{lemma}

\begin{proof}
$(1)$ Notice that $i_{\ast}\mathcal{A} \subseteq \mathcal{V} \subseteq \mathcal{B}$. It is evident that for any $V \in \mathcal{V}$, we have $i^{\ast}V \in \mathcal{A}$. This implies that $i_{\ast}i^{\ast}\mathcal{V} \subseteq \mathcal{V}$. The proof for $i_{\ast}i^{!}\mathcal{V} \subseteq \mathcal{V}$ follows similarly.

$(2)$ For each object $V \in \mathcal{V} \subseteq \mathcal{B}$, it follows from (R4) that there exists a commutative diagram
\begin{equation*}\label{three1}
\xymatrix{
  &i_\ast i^! V \ar[r]^-{\theta_V}&V\ar[rr]^-{\vartheta_V}\ar[dr]_{h_{1}}&  &j_\ast j^\ast V \ar[r]&i_\ast A_{1} &\\
           &                &       &  M \ar[ur]_{h_{2}}& }
\end{equation*}
with $A_{1} \in \mathcal{A}$, where $\theta_V$ and $\vartheta_V$ are given by the adjunction morphisms, such that $i_\ast i^! V\stackrel{\theta_V}{\longrightarrow}V\stackrel{h_{1}}{\longrightarrow}M\stackrel{}\dashrightarrow$ and $M\stackrel{h_{2}}{\longrightarrow}j_\ast j^\ast V\stackrel{ }{\longrightarrow}i_\ast A_{1}\stackrel{}\dashrightarrow$ are $\mathbb{E}_{\B}$-triangles. From the first $\mathbb{E}_{\B}$-triangle, we conclude that $M \in \mathcal{V}$, which follows directly from the definition of a thick subcategory, since $\mathcal{V}$ is a thick subcategory of $\mathcal{B}$ and $i_\ast i^! V \in \mathcal{V}$. Now, considering the second $\mathbb{E}_{\B}$-triangle, we have that $i_\ast A_{1} \in \mathcal{V}$, as $i_{\ast}\mathcal{A} \subseteq \mathcal{V}$. Since $\mathcal{V}$ is closed under extensions, we conclude that $j_{\ast}j^{\ast}V \in \mathcal{V}$ and thus $j_{\ast}j^{\ast}\mathcal{V} \subseteq \mathcal{V}$. Similarly, we can prove $j_{!}j^{\ast}\mathcal{V} \subseteq \mathcal{V}$.


$(3)$ By hypothesis and (R4), we can construct a commutative diagram
\begin{equation*}\label{three2}
\xymatrix{
  &i_\ast i^! V \ar[r]^-{\theta_V}&V\ar[rr]^-{\vartheta_V}\ar[dr]_{k_{1}}&  &j_\ast j^\ast V \ar[r]&i_\ast A_{2} &\\
           &                &       &  K \ar[ur]_{k_{2}}& }
\end{equation*}
with $A_{2} \in \mathcal{A}$, such that the sequences $i_\ast i^! V\stackrel{\theta_V}{\longrightarrow}V\stackrel{k_{1}}{\longrightarrow}K\stackrel{}\dashrightarrow$ and $K\stackrel{k_{2}}{\longrightarrow}j_\ast j^\ast V\stackrel{ }{\longrightarrow}i_\ast A_{2}\stackrel{}\dashrightarrow$ are $\mathbb{E}_{\B}$-triangles. By (2), we have $j_{\ast}j^{\ast}V \in j_{\ast}j^{\ast}\mathcal{V} \subseteq \mathcal{V}$. Since $i_\ast A_{2} \in i_\ast \mathcal{A} \subseteq \mathcal{V}$ and  $\mathcal{V}$ is closed under cocones, we have  $K \in \mathcal{V}$. Moreover, since $i_\ast i^! V \in i_\ast \mathcal{A} \in \mathcal{V}$, it follows that $V \in \mathcal{V}$, given that the thick subcategory $\mathcal{V}$ is closed under extensions.
\end{proof}


For convenience, we denote by $\Thick\mathcal{C}$ the set of thick subcategories of $\mathcal{C}$, and by $\Thick_{i_{\ast}\mathcal{A}}\mathcal{B}$ the set of thick subcategories of $\mathcal{B}$ that contain $i_{\ast}\mathcal{A}$. Now we state our first main result of this section.

\begin{theorem}\label{main-thick-bijection}
There exist mutually inverse bijections
$$
\xymatrix@C=3.5pc{{\Thick_{i_{\ast}\mathcal{A}}\mathcal{B}}\ar@<-1ex>[r]_-{\Phi }&
{\Thick\mathcal{C} }\ar@<-1ex>[l]_-{\Psi}.        }
$$
The map $\Phi$ is defined by $\mathcal{V} \mapsto j^{\ast}\mathcal{V}$ and the inverse is given by $\Psi: \mathcal{W} \mapsto
\mathcal{V} = \{M \in \mathcal{B} \mid j^{\ast}M \in \mathcal{W}\}$.
\end{theorem}

For a morphism $f:A\rightarrow B$, we denote by $\Phi_{f}$ the set consisting of all pairs $(h_{1},h_{2})$ such that $h_{1}:B\rightarrow K$ is a deflation, $h_{2}:K\rightarrow C$ is an inflation and $f=h_{2}h_{1}$.

\begin{lemma}\label{lemma 32}
If $\mathcal{V}\in\Thick_{i_{\ast}\mathcal{A}}\mathcal{B}$, then $\Phi(\mathcal{V})\in\Thick\mathcal{C}$.
\end{lemma}

\begin{proof} Consider an $\mathbb{E}_{\C}$-triangle $X\stackrel{f}{\longrightarrow}Y\stackrel{g}{\longrightarrow}Z\stackrel{}\dashrightarrow$ in $\mathcal{C}$. Since the functor $j_{\ast}$ is left exact, there exists an $\mathbb{E}_{\B}$-triangle
\begin{equation}\label{three3}
\xymatrix{
j_{\ast}X\stackrel{j_{\ast}f}{\longrightarrow}j_{\ast}Y\stackrel{h_{1}}{\longrightarrow} Z'\stackrel{}\dashrightarrow}
\end{equation}
and a commutative diagram
\begin{equation*}
\xymatrix{
   &   & j_{\ast}Z & \\
j_{\ast}X\ar[r]^{j_{\ast}f}  & j_{\ast}Y \ar[ur]^{j_{\ast}g }\ar[r]^{h_{1}} & Z'\ar[u]_{h_{2}}  \ar@{-->}[r]^{ } & }
\end{equation*}
where $h_{2}:Z'\rightarrow j_{\ast}Z$ is a compatible inflation such that $(h_{1},h_{2})\in\Phi_{ j_{\ast}g}$. We claim that $Z\cong j^\ast Z'$.
By Lemma \ref{CY}(1), $j^{\ast}j_{\ast}\cong \Id_{\C}$, then
$j^{\ast}j_{\ast}X\stackrel{j^{\ast}j_{\ast}f}{\longrightarrow}j^{\ast}j_{\ast}Y \stackrel {j^{\ast}j_{\ast}g}{\longrightarrow}j^{\ast}j_{\ast}Z\stackrel{}\dashrightarrow$ is also an $\mathbb{E}_{\C}$-triangle. Since $j^\ast j_\ast g=(j^\ast h_2)(j^\ast h_1)$ is a deflation and by Condition \ref{WIC}, we obtain that $j^{\ast}h_{2}$ is a deflation. Consequently, $j^{\ast}h_{2}$ is an isomorphism since $j^{\ast}h_{2}:j^\ast Z'\rightarrow j^*j_\ast Z$ is an inflation and compatible. Hence, we have $ j^\ast Z'\cong  j^*j_\ast Z \cong Z$.

Now, we prove that $\Phi(\mathcal{V}) = j^{\ast}\mathcal{V}$ is closed under extensions, cones and cocones. Assume that $X \in j^{\ast}\mathcal{V}$ and $Z \in j^{\ast}\mathcal{V}$. It follows from Lemma \ref{lemma 31}(2) that $j_{\ast}X \in j_{\ast}j^{\ast}\mathcal{V} \subseteq \mathcal{V}$. Since $j^\ast Z' \cong Z \in j^{\ast}\mathcal{V}$, by Lemma \ref{lemma 31}(3), we conclude that $Z' \in \mathcal{V}$.
Then the $\mathbb{E}_{\B}$-triangle \ref{three3} satisfying $j_{\ast}X \in \mathcal{V}$ and $Z' \in \mathcal{V}$.
Therefore, $j_{\ast}Y \in \mathcal{V}$ because the thick subcategory $\mathcal{V}$ is closed under extensions. By Lemma \ref{CY}(1), we have $Y \cong j^\ast j_{\ast} Y \in j^{\ast}\mathcal{V}$, which implies that $j^{\ast}\mathcal{V}$ is closed under extensions.

Assume that $X \in j^{\ast}\mathcal{V}$ and $Y \in j^{\ast}\mathcal{V}$. By Lemma \ref{lemma 31}(2), we have $j_{\ast}X \in j_{\ast}j^{\ast}\mathcal{V} \subseteq \mathcal{V}$ and $j_{\ast}Y \in j_{\ast}j^{\ast}\mathcal{V} \subseteq \mathcal{V}$. Therefore, $Z' \in \mathcal{V}$ because $\mathcal{V}$ is closed under cones. Thus, we conclude that $Z \cong j^\ast Z' \in j^{\ast}\mathcal{V}$, which implies that $j^{\ast}\mathcal{V}$ is closed under cones.

Assume that $Y \in j^{\ast}\mathcal{V}$ and $Z \in j^{\ast}\mathcal{V}$. Since $j^\ast Z' \cong Z \in j^{\ast}\mathcal{V}$, by Lemma \ref{lemma 31}(3), we conclude that $Z' \in \mathcal{V}$. By Lemma \ref{lemma 31}(2), we have $j_{\ast}Y \in j_{\ast}j^{\ast}\mathcal{V} \subseteq \mathcal{V}$.
Then the $\mathbb{E}_{\B}$-triangle (\ref{three3}) satisfying $j_{\ast}Y \in \mathcal{V}$ and $Z' \in \mathcal{V}$.
Therefore, $j_{\ast}X \in \mathcal{V}$ because the thick subcategory $\mathcal{V}$ is closed under cocones. By Lemma \ref{CY}(1), we have $X \cong j^\ast j_{\ast} X \in j^{\ast}\mathcal{V}$, which implies that $j^{\ast}\mathcal{V}$ is closed under cocones.

It remains to show that $j^{\ast}\mathcal{V}$ is closed under direct summands. Let $M = M_{1} \oplus M_{2} \in j^{\ast}\mathcal{V}$. Since $j^{\ast}$ is dense, there exist objects $B_{1}, B_{2}$ in $\mathcal{B}$ such that $M_{1} \cong j^{\ast}B_{1}$ and $M_{2} \cong j^{\ast}B_{2}$. Observe that
$j^{\ast}(B_{1} \oplus B_{2}) \cong j^{\ast}B_{1} \oplus j^{\ast}B_{2} \cong M_{1} \oplus M_{2} = M$
belongs to $j^{\ast}\mathcal{V}$. Then Lemma \ref{lemma 31}(3) implies that $B_{1} \oplus B_{2}\in\mathcal{V}$. Since  $\mathcal{V}$ is closed under direct summands, it follows that $B_{1}$ and $B_{2}\in\mathcal{V}$. Therefore, we have $j^{\ast}B_{1}, j^{\ast}B_{2} \in j^{\ast}\mathcal{V}$ and  conclude that $j^{\ast}\mathcal{V}$ is closed under direct summands.
\end{proof}


\begin{lemma}\label{prop 33} If $\mathcal{W}\in\Thick\mathcal{C}$, then $\Psi(\mathcal{W})\in\Thick_{i_{\ast}\mathcal{A}}\mathcal{B}$.
\end{lemma}

\begin{proof}
Consider an $\mathbb{E}_{\B}$-triangle $X\stackrel{f}{\longrightarrow}Y\stackrel{g}{\longrightarrow}Z\stackrel{}\dashrightarrow$ in $\mathcal{B}$. Since $j^{\ast}$ is exact, we obtain an $\mathbb{E}_{\C}$-triangle $j^{\ast}X\stackrel{j^{\ast}f}{\longrightarrow}j^{\ast}Y\stackrel{j^{\ast}g}{\longrightarrow}j^{\ast}Z\stackrel{}\dashrightarrow$. Assume that $X, Z \in \Psi(\mathcal{W})$, then $j^{\ast}X, j^{\ast}Z \in \mathcal{W}$.  Since $\mathcal{W}$ is closed under extensions, we have that $j^{\ast}Y\in\mathcal{W}$ and hence $Y\in\mathcal{V}$. It means that $\mathcal{V}$ is closed under extensions. The proof that $\mathcal{V}$ is closed under cones and cocones follows similarly. We still need to show that $\mathcal{V}$ is closed under direct summands. For $M = M_{1} \oplus M_{2} \in \Psi(\mathcal{W})$, we have $j^{\ast}M \in \mathcal{W}$ and hence $j^{\ast}M_{1} \oplus j^{\ast}M_{2} \cong j^{\ast}(M_{1} \oplus M_{2}) = j^{\ast}M\in\mathcal{W}$. Since $\mathcal{W}$ is closed under direct summands, it follows that $j^{\ast}M_{1},j^{\ast}M_{2}\in\mathcal{W}$. This implies that $M_{1}$ and $M_{2}$ belong to $\mathcal{V}$.
\end{proof}

Now we are in the position to prove Theorem \ref{main-thick-bijection}.


\emph{Proof of Theorem \ref{main-thick-bijection}.}
From Lemma \ref{lemma 32}, we conclude that $\Phi$ is well-defined, and from Lemma \ref{prop 33}, we see that $\Psi$ is well-defined as well. It remains to show that the maps $\Phi$ and $\Psi$ are inverses of each other. Assume that $\mathcal{W} \in \Thick\mathcal{C}$. Recall that $\Psi(\mathcal{W})=\{M \in \mathcal{B} \mid j^{\ast}M \in \mathcal{W}\}$. It is evident that $\Phi\Psi(\mathcal{W})= j^{\ast}\Psi(\mathcal{W}) \subseteq \mathcal{W}$. On the other hand, for any $N \in \mathcal{W}$, by Lemma \ref{CY}(1), we have $j^{\ast}j_{\ast}N \cong N \in \mathcal{W}$. This implies that $j_{\ast}N\in\Psi(\mathcal{W})$, and therefore $N \in \Phi\Psi(\mathcal{W})$. Conversely, for $\mathcal{V} \in \Thick_{i_{\ast}\mathcal{A}}\mathcal{B}$, we have $\Psi\Phi(\mathcal{V})=\{M \in \mathcal{B} \mid j^{\ast}M \in j^{\ast}\mathcal{V}\}$. The inclusion $V\subseteq\Psi\Phi(\mathcal{V})$ is obvious. The inclusion $\Psi\Phi(\mathcal{V})\subseteq V$ immediately follows from Lemma \ref{lemma 31}(3). We complete the proof.
 \fin\\

We provide an example illustrating Theorem \ref{main-thick-bijection}.

\begin{example}\label{fang}
Let $A$ be the path algebra of the quiver $1\stackrel{\alpha}\longrightarrow2$ over a field.
Then the triangular matrix algebra $B=\begin{pmatrix} A &A\\0  &A\end{pmatrix}$ is given by the quiver
$$
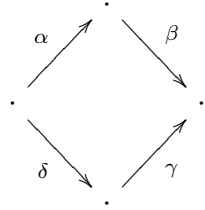
\xymatrix{
   &   \cdot\ar[dr]^{\beta} &  &  \\
   \cdot\ar[ur]^{\alpha}\ar[dr]_{\delta} &    &  \cdot& \\
    &  \cdot \ar[ur]_{\gamma} &  & \\  }
$$
with the relation $\beta\alpha=\gamma\delta$. Each $B$-module can be written as a triple ${X\choose Y}_f$ such that $f:Y\rightarrow X$ is a homomorphism of $A$-modules. For convience, we write ${X\choose Y}$ instead of ${X\choose Y}_0$. The Auslander-Reiten quiver of $B$ is given by:
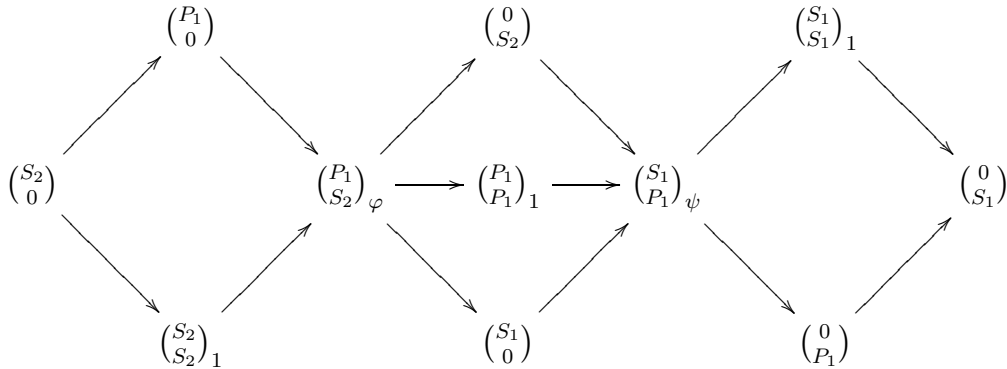
\begin{equation*}
\xymatrix@!=3.0pc{    & P_1\choose0\ar[dr] &   & 0\choose S_2 \ar[dr] &      & {S_1\choose S_1}_{1}\ar[dr]   &   \\
S_2\choose 0 \ar[dr]\ar[ur]&    & {P_1\choose S_2}_{\varphi}\ar[dr]\ar[ur]\ar[r]  & {P_1\choose P_1}_{1}\ar[r]  &  {S_1\choose P_1}_{\psi}\ar[ur]\ar[dr]    &    & 0\choose S_1  \\
 &  {S_2\choose S_2}_{1}\ar[ur]  &&  S_1\choose 0 \ar[ur]      &   &0\choose P_1\ar[ur] }
\end{equation*}

(1) By \cite[Example 2.12]{Ps}, we have a recollement of abelian categories
\begin{equation*}
  \xymatrix{\mod A\ar[rr]|{i_{*}}&&\ar@/_1pc/[ll]|{i^{*}}\ar@/^1pc/[ll]|{i^{!}}\mod B
\ar[rr]|{j^{\ast}}&&\ar@/_1pc/[ll]|{j_{!}}\ar@/^1pc/[ll]|{j_{\ast}} \mod A}
\end{equation*}
such that $i^{*}({X\choose Y}_{f})=\rm Coker$$(f)$, $i_{*}(X)={X\choose 0}$, $i^{!}({X\choose Y}_{f})=X$, $j_{!}(Y)={Y\choose Y}_{1}$, $j^{\ast}({X\choose Y}_{f})=Y$ and  $j_{*}(Y)={0\choose Y}$. By Theorem \ref{main-thick-bijection}, there exist a bijection between thick subcategories in $\mod A$ and thick subcategories in $\mod B$ containing
$$i_{\ast}\mod A=\add({S_2\choose0}, {P_1\choose0}, {{S_1}\choose0}).$$
Let us list all thick subcategories $\mathcal{V}$ in $\mod B$ and the corresponding thick subcategories $\Phi(\mathcal{V})$ in $\mod A$ as follows:
$$
\begin{tabular}{|p{9.5cm}|p{4.5cm}|}
\hline
$\mathcal{V}$ & $\Phi(\mathcal{V})$\\
 \hline
$\add({S_2\choose0}\oplus{P_1\choose 0}\oplus{S_1\choose 0})$ & 0\\
 \hline
$\add({S_2\choose0}\oplus{P_1\choose 0}\oplus{S_1\choose 0}\oplus{P_1\choose S_2}_{\varphi}\oplus{S_2\choose S_2}_{1}\oplus {0\choose S_2})$ & $\add({S_2})$\\
 \hline
$\add({S_2\choose0}\oplus{P_1\choose 0}\oplus{S_1\choose 0}\oplus{S_1\choose S_1}_{1}\oplus {0\choose S_1})$ & $\add({S_1})$\\
 \hline
$\add({S_2\choose0}\oplus{P_1\choose 0}\oplus{S_1\choose 0}\oplus{P_1\choose P_1}_{1}\oplus{S_1\choose P_1}_{\psi}\oplus{0\choose P_1})$ & $\add({P_1})$\\
 \hline
$\mod B$ & $\mod A$\\
 \hline
\end{tabular}
$$
\end{example}
%


\bigskip

In the remainder of this section, we show that the thick subcategories $\mathcal{V}$ in $\mathcal{B}$ containing $i_{\ast}\mathcal{A}$ can induce a new recollement.

\begin{lemma}\label{middle to left}
 Let $\mathcal{V}$ be a thick subcategory of $\mathcal{B}$ such that $i_{\ast}i^{\ast}\mathcal{V} \subseteq \mathcal{V}$ (resp. $i_{\ast}i^{!}\mathcal{V}\subseteq \mathcal{V}$). Then $i^{\ast}\mathcal{V}$ (resp. $i^{!}\mathcal{V}$) is a thick subcategory of $\mathcal{A}$.
\end{lemma}

\begin{proof} We prove only the first statement; the second one can be shown similarly.

Consider an $\mathbb{E}_{\A}$-triangle $X\stackrel{f}{\longrightarrow}Y\stackrel{g}{\longrightarrow}Z\stackrel{}\dashrightarrow$ in $\mathcal{A}$ where $X, Z \in i^{\ast}\mathcal{V}$. Since $i_{\ast}$ is exact, we obtain an $\mathbb{E}_{\B}$-triangle $i_{\ast}X\stackrel{f}{\longrightarrow}i_{\ast}Y\stackrel{g}{\longrightarrow}i_{\ast}Z\stackrel{}\dashrightarrow$. By hypothesis, we have that $i_{\ast}X, i_{\ast}Z \in i_{\ast}i^{\ast}\mathcal{V} \subseteq \mathcal{V}$. Since  $\mathcal{V}$ is closed under extensions, it follows that $i_{\ast}Y \in \mathcal{V}$. Thus, we conclude that $Y \cong i^\ast i_{\ast} Y \in i^{\ast}\mathcal{V}$. This implies that $i^{\ast}\mathcal{V}$ is closed under extensions. The proof that $i^{\ast}\mathcal{V}$ is closed under cones and cocones follows similarly. Now, we show that $i^{\ast}\mathcal{V}$ is closed under direct summands. Let $M = M_{1} \oplus M_{2} \in i^{\ast}\mathcal{V}$. It is evident that $i_{\ast}M_{1} \oplus i_{\ast}M_{2} \cong i_{\ast}(M_{1} \oplus M_{2}) = i_{\ast}M \in i_{\ast}i^{\ast}\mathcal{V} \subseteq \mathcal{V}$. Since the thick subcategory $\mathcal{V}$ is closed under direct summands, it follows that $i_{\ast}M_{1}$ and $i_{\ast}M_{2}$ belong to $\mathcal{V}$. By Lemma \ref{CY}(1), we have $M_{1} \cong i^{\ast}i_{\ast}M_{1}$ and $M_{2} \cong i^{\ast}i_{\ast}M_{2}$, which implies that both $M_{1}$ and $M_{2}$ belong to $i^{\ast}\mathcal{V}$. Then $i^{\ast}\mathcal{V}$ is closed under direct summands and hence $i^{\ast}\mathcal{V}$ is a thick subcategory of $\mathcal{A}$.
\end{proof}

\begin{remark} \label{remark}
Notice that $i_{\ast}i^{\ast}\mathcal{V} \subseteq \mathcal{V}$ (resp. $i_{\ast}i^{!}\mathcal{V}\subseteq \mathcal{V}$) is weaker than  $i_{\ast}\mathcal{A}\subseteq \mathcal{V}$.
In fact, $i_{\ast}\mathcal{A}\subseteq \mathcal{V}$ implies that $i_{\ast}i^{\ast}\mathcal{V}\subseteq i_{\ast}\mathcal{A}\subseteq\mathcal{V}$ and $i_{\ast}i^{!}\mathcal{V}\subseteq i_{\ast}\mathcal{A}\subseteq \mathcal{V}$.
By Lemma \ref{middle to left}, we have that $i^{\ast}\mathcal{V}$ and $i^{!}\mathcal{V}$ are thick subcategories of $\mathcal{A}$.
Since $i_{\ast}\mathcal{A} \subseteq \mathcal{V}$, by Lemma \ref{CY}(1), then we have $\mathcal{A} \cong i^{\ast}i_{\ast}\mathcal{A} \subseteq i^{\ast}\mathcal{V}$ and $\mathcal{A} \cong i^{!}i_{\ast}\mathcal{A} \subseteq i^{!}\mathcal{V}$. Therefore, $i^{\ast}\mathcal{V} = i^{!}\mathcal{V} = \mathcal{A}$ holds.

\end{remark}
%

\begin{lemma}\label{adjoint-restrict}
Let $(F, G)$ be an adjoint pair from $\mathcal{A}$ to $\mathcal{B}$. If $\mathcal{D}$ is a subcategory of $\mathcal{A}$ satisfying $GF(\mathcal{D}) \subseteq \mathcal{D}$, then the restricted functors $(\overline{F}, \overline{G}) = (F|_{\mathcal{D}}, G|_{F(\mathcal{D})})$ is an adjoint pair from $\mathcal{D}$ to $F(\mathcal{D})$. Dually, if $\mathcal{D'}$ is a subcategory of $\mathcal{B}$ satisfying $FG(\mathcal{D'}) \subseteq \mathcal{D'}$, then the restricted functors $(\overline{F}, \overline{G}) = (F|_{G(\mathcal{D'})}, G|_{\mathcal{D'}})$ is an adjoint pair from $G(\mathcal{D'})$ to $\mathcal{D'}$.
\end{lemma}

\begin{proof}
For any $D_{1}, D_{2} \in \mathcal{D}$, by the definition of the functors $\overline{F}$ and $\overline{G}$, we have
$$\Hom_{F(\mathcal{D})}(\overline{F}(D_{1}), F(D_{2})) = \Hom_{\mathcal{B}}(F(D_{1}), F(D_{2})).$$
Since $(F, G)$ is an adjoint pair from $\mathcal{A}$ to $\mathcal{B}$, we have the following bijection:
$$\Hom_{\mathcal{B}}(F(D_{1}), F(D_{2})) \cong \Hom_{\mathcal{A}}(D_{1}, GF(D_{2})).$$
Noting that $GF(\mathcal{D}) \subseteq \mathcal{D}$, we obtain:
$$\Hom_{\mathcal{A}}(D_{1}, GF(D_{2})) = \Hom_{\mathcal{D}}(D_{1}, GF(D_{2})),$$
which is functorial in $D_{1}$ and $F(D_{2})$. Therefore, $(\overline{F}, \overline{G})$ is an adjoint pair from $\mathcal{D}$ to $F(\mathcal{D})$. Similarly, we can prove $(\overline{F}, \overline{G})$ is an adjoint pair from $G(\mathcal{D'})$ to $\mathcal{D'}$.
\end{proof}

\begin{proposition}\label{main-thick} If $\mathcal{V}\in\Thick_{i_{\ast}\mathcal{A}}\mathcal{B}$, then $\mathcal{V}$ induce a recollement as follows:
\begin{equation*}
  \xymatrix{\mathcal{A}\ar[rr]|{\overline{i_{*}}}&&\ar@/_1pc/[ll]|{\overline{i^{*}}}\ar@/^1pc/[ll]|{\overline{i^{!}}}\mathcal{V}
\ar[rr]|{\overline{j^{\ast}}}&&\ar@/_1pc/[ll]|{\overline{j_{!}}}\ar@/^1pc/[ll]|{\overline{j_{\ast}}}j^{\ast}\mathcal{V}}
\end{equation*}
\end{proposition}

\begin{proof} By Remark \ref{remark}, $i^{\ast}\mathcal{V} = i^{!}\mathcal{V} = \mathcal{A}$ is thick.
By Lemma \ref{lemma 32}, $j^{\ast}\mathcal{V}$ is a thick subcategory in $\mathcal{C}$. It is clear that $\overline{i_{*}}, \overline{j^{\ast}}$ are exact functors, $\overline{i^{\ast}}$, $\overline{j_!}$ are right exact functors and $\overline{i^{!}}$, $\overline{j_\ast}$ are left exact functors. It follows from Lemma \ref{lemma 31}(1) and Lemma \ref{adjoint-restrict} that $(\overline{i^{*}}, \overline{i_{\ast}}, \overline{i^{!}})$ and $(\overline{j_!}, \overline{j^\ast}, \overline{j_\ast})$ are adjoint triples. Hence (R1) holds. It is evident that $\Im \overline{i_{\ast}}=\Ker \overline{j^{\ast}}$, then (R2) holds. From (R3), we know that $i_\ast$, $j_!$, and $j_\ast$ are fully faithful. Thus, the restriction functors $\overline{i_\ast}$, $\overline{j_!}$, and $\overline{j_\ast}$ are also fully faithful. Therefore, (R3) holds. We obtain (R4) and (R5) immediately, and the proof is complete.
\end{proof}

We provide an example illustrating Proposition \ref{main-thick} to end this section.

\begin{example} Keep the notations used in Example \ref{fang}.

Set $\mathcal{X}_{1}=\mod A$, $\mathcal{X}_{2}=\add({S_1\oplus P_1})$ and
$$\mathcal{X}=\add({S_2\choose0}\oplus{P_1\choose 0}\oplus{S_1\choose 0}\oplus  {P_1\choose P_1}_{1}\oplus {S_1\choose P_1}_{\psi}\oplus{S_1\choose S_1}_{1}\oplus {0\choose P_1}\oplus {0\choose S_1}),$$
By \cite[Example 3.5]{WL}, we have a recollement of extriangulated categories which is neither abelian nor triangulated as follows:
\begin{equation*}
  \xymatrix{\mathcal{X}_{1}\ar[rr]|{i_{*}}&&\ar@/_1pc/[ll]|{i^{*}}\ar@/^1pc/[ll]|{i^{!}}\mathcal{X}
\ar[rr]|{j^{\ast}}&&\ar@/_1pc/[ll]|{j_{!}}\ar@/^1pc/[ll]|{j_{\ast}} \mathcal{X}_{2}}
\end{equation*}

Let $\mathcal{X}_{1}=\mod A$, and
$$\mathcal{V}=\add({S_2\choose0}\oplus{P_1\choose 0}\oplus{S_1\choose 0} \oplus{S_1\choose S_1}_{1}\oplus {0\choose S_1})\in\Thick_{i_{\ast}\mathcal{X}_{1}}\mathcal{X},$$
then $j^{\ast}\mathcal{V} = \add({S_1})$ is a thick subcategory of $\mathcal{X}_{2}$.
It is straightforward to check that

\begin{equation*}
  \xymatrix{\mathcal{X}_{1}\ar[rr]|{\overline{i_{*}}}&&\ar@/_1pc/[ll]|{\overline{i^{*}}}\ar@/^1pc/[ll]|{\overline{i^{!}}}\mathcal{V}
\ar[rr]|{\overline{j^{\ast}}}&&\ar@/_1pc/[ll]|{\overline{j_{!}}}\ar@/^1pc/[ll]|{\overline{j_{\ast}}}j^{\ast}\mathcal{V}}
\end{equation*}
is a recollement of thick subcategories.

\end{example}

\section{Gluing of silting subcategories}
%

\begin{definition}\label{hcotors} \text{\rm (\cite[Definition 3.1]{AT})}
Let $\mathcal{T}$ and $\mathcal{F}$ be two subcategories of the extriangulated category $\mathscr{C}$.
\begin{itemize}
  \item [(1)] The pair $(\mathcal{T},\mathcal{F})$ is called a {\em cotorsion pair} in $\mathscr{C}$ if it satisfies the following conditions:
      \begin{itemize}
          \item [(CP1)] $\mathcal{T}$ and $\mathcal{F}$ are closed under direct sunmmands.
          \item [(CP2)] $\mathbb{E}(\mathcal{T},\mathcal{F})=0.$
          \item [(CP3)] For any $C\in\mathscr{C}$, there exists a conflation $F\longrightarrow T\longrightarrow C$ such that $F\in\mathcal{F}$, $T\in\mathcal{T}$.
          \item [(CP4)] For any $C\in\mathscr{C}$, there exists a conflation $C\longrightarrow F'\longrightarrow T'$ such that $F'\in\mathcal{F}$, $T'\in\mathcal{T}$.
      \end{itemize}
  \item [(2)] A cotorsion pair $(\mathcal{T},\mathcal{F})$ is called a {\em hereditary cotorsion pair} if it satisfies the following condition.
      \begin{itemize}
          \item [(HCP)] $\mathbb{E}^{k}(\mathcal{T},\mathcal{F})=0$ for each $k \geq 2$.
      \end{itemize}
\end{itemize}
\end{definition}

\begin{definition}(\cite[Definition 4.1]{AT})
Let $\mathscr{C}$ be an extriangulated category and $\mathcal{X}$ be a subcategory of $\mathscr{C}$. For each $n \geq 0$, we inductively define subcategories $$\mathcal{X}_{n}^{\wedge} := \Cone(\mathcal{X}_{n-1}^{\wedge}, \mathcal{X})~\text{and}~\mathcal{X}_{n}^{\vee} := \Cocone(\mathcal{X}, \mathcal{X}_{n-1}^{\vee}),$$
where $\mathcal{X}_{-1}^{\wedge} := \{0\}$ and $\mathcal{X}_{-1}^{\vee} := \{0\}$. Put
$$\mathcal{X}^{\wedge} := \bigcup \mathcal{X}_{n}^{\wedge},~~\mathcal{X}^{\vee} := \bigcup \mathcal{X}_{n}^{\vee}.$$
\end{definition}

We say  a cotorsion pair $(\mathcal{T},\mathcal{F})$ in $\mathscr{C}$ is {\em bounded} if $ \mathcal{T}^{\wedge}=\mathscr{C}=\mathcal{F}^{\vee}$ (See \cite[Section 4]{AT}).

\begin{definition}\label{silt-subcat}(\cite[Definition 5.1]{AT})
Assume that $\mathcal{M}$ is a subcategory of the extriangulated category $\mathscr{C}$. We call $\mathcal{M}$ a {\em silting subcategory} of $\mathscr{C}$ if it satisfies the following conditions.
\begin{itemize}
  \item [(1)] $\mathcal{M}$ is a {\em presilting subcategory}, i.e., $\mathcal{M} = \add\mathcal{M}$ and $\mathbb{E}^{k}(\mathcal{M}, \mathcal{M}) = 0$ for all $k \geq 1$.
  \item [(2)] $\mathscr{C} = \text{thick}\mathcal{M}$.
\end{itemize}
\end{definition}

Let $\text{bbd-hcotors}\mathscr{C}$ denote the set of bounded hereditary cotorsion pairs of $\mathscr{C}$, and let $\text{silt}\mathscr{C}$ denote the set of silting subcategories of $\mathscr{C}$. By \cite[Theorem 5.7]{AT}, there exists a bijection
$$
\xymatrix@C=3.5pc{{\text{bbd-hcotors}\mathscr{C}}\ar@<-1ex>[r]_-{\Phi }&
{\text{silt}\mathscr{C} }\ar@<-1ex>[l]_-{\Psi},        }
$$
where the map $\Phi$ is given by $\Phi(\mathcal{X},\mathcal{Y}) :=  \mathcal{X} \cap \mathcal{Y}$ and the inverse given by
$\Psi(\mathcal{M}) := (\mathcal{M}^{\vee}, \mathcal{M}^{\wedge})$.  Therefore, the silting subcategories  is often identified with bounded hereditary cotorsion pairs.

In what follows,  we assume that  $(\mathcal{A}, \mathcal{B}, \mathcal{C}, i^{*}, i_{\ast}, i^{!},j_!, j^\ast, j_\ast)$ is a recollement of extriangulated categories as in (\ref{recolle}). We need the following lemmas.

\begin{lemma}\label{H-extension}\text{\rm (\cite[Lemma 4.3]{MZ})}
Let $F:\mathcal{A}\rightarrow \mathcal{B}$ be an exact functor that admitting a right adjoint functor $G$.
For any $X\in\mathcal{A}, Y\in\mathcal{B}$, and for any $k \geq 1$, if one of the following conditions is satisfied
\begin{itemize}
  \item [(1)] If $F$ is an exact functor and preserves projective objects;
  \item [(2)] If $G$ is an exact functor and preserves injective objects;
\end{itemize}
then we have
$$\mathbb{E}^{k}_{\mathcal{B}}(FX,Y)\cong\mathbb{E}^{k}_{\mathcal{A}}(X,GY).$$
\end{lemma}

\begin{lemma}\label{prop 311} For any positive integer $n\geq1$, the following holds.
\begin{itemize}
  \item [(1)] If $\mathcal{B}$ has enough projective objects and $i^{\ast}$ is exact. Then  $\mathbb{E}^{n}_{\mathcal{A}}(i^{\ast}B, A) \cong \mathbb{E}^{n}_{\mathcal{B}}(B, i_{\ast}A)$ for any $A \in \mathcal{A}$ and $B \in \mathcal{B}$.
  \item [(2)] If $\mathcal{A}$ has enough projective objects and $i^{!}$ is exact. Then  $\mathbb{E}^{n}_{\mathcal{B}}(i_{\ast}A, B) \cong \mathbb{E}^{n}_{\mathcal{A}}(A, i^{!}B)$ for any $A \in \mathcal{A}$ and $B \in \mathcal{B}$.
  \item [(3)] If $\mathcal{C}$ has enough projective objects and $j_{!}$ is exact. Then  $\mathbb{E}^{n}_{\mathcal{B}}(j_{!}C, B) \cong \mathbb{E}^{n}_{\mathcal{C}}(C, j^{\ast}B)$ for any $C \in \mathcal{C}$ and $B \in \mathcal{B}$.
  \item [(4)] If $\mathcal{B}$ has enough projective objects and $j_{\ast}$ is exact. Then
$\mathbb{E}^{n}_{\mathcal{C}}(j^{\ast}B, C) \cong \mathbb{E}^{n}_{\mathcal{B}}(B, j_{\ast}C)$ for any $C \in \mathcal{C}$ and $B \in \mathcal{B}$.
\end{itemize}
\end{lemma}

\begin{proof}
We will only prove (1), the proof of (2), (3), and (4) are similar.
By (R1), $(i^{\ast}, i_{\ast})$ is an adjoint pair. Applying Lemma \ref{CY}(3), we have that $i^{\ast}$ preserves projective objects. Since $i^{\ast}$ is exact, by Lemma \ref{H-extension}, then $\mathbb{E}^{n}_{\mathcal{A}}(i^{\ast}B, A) \cong \mathbb{E}^{n}_{\mathcal{B}}(B, i_{\ast}A)$ holds for any $A \in \mathcal{A}$ and $B \in \mathcal{B}$.
\end{proof}

To present the second main result of this paper, we need the following notion from \cite{WL}.

\begin{definition}\label{glued}  \text{\rm (\cite[Definition 4.3]{WL})}
Given cotorsion pairs $(\mathcal{T}_{1},\mathcal{F}_{1})$ and $(\mathcal{T}_{2},\mathcal{F}_{2})$ in $\mathcal{A}$ and $\mathcal{C}$, respectively, we define
\begin{equation*}
\mathcal{T}=\{B\in \mathcal{B}~|~i^{\ast }B\in\mathcal{T}_{1}~\text{and}~j^{\ast}B\in \mathcal{T}_{2}  \}
\end{equation*}
\begin{equation*}
\mathcal{F}=\{B\in \mathcal{B}~|~i^{!}B\in\mathcal{F}_{1}~\text{and}~j^{\ast}B\in \mathcal{F}_{2}  \}.
\end{equation*}
In this case, we refer to $(\mathcal{T},\mathcal{F})$ as the {\em glued pair} with respect to $(\mathcal{T}_{1},\mathcal{F}_{1})$ and $(\mathcal{T}_{2},\mathcal{F}_{2})$.
\end{definition}

\begin{lemma}\label{main-silt-lemma}
Let $(\mathcal{T}_{1},\mathcal{F}_{1})$ and $(\mathcal{T}_{2},\mathcal{F}_{2})$ be two bounded hereditary cotorsion pairs in $\mathcal{A}$ and $\mathcal{C}$, respectively. Let $(\mathcal{T},\mathcal{F})$ be the glued pair with respect to $(\mathcal{T}_{1},\mathcal{F}_{1})$ and $(\mathcal{T}_{2},\mathcal{F}_{2})$.
Suppose that  $i^{!}, i^{\ast}$ are exact and $M \in \mathcal{B}$.
\begin{itemize}
  \item[(1)] There exists an $\mathbb{E}_\mathcal{B}$-triangle
$K_{n-1} \stackrel{}{\longrightarrow} T_{n}\stackrel{}{\longrightarrow} M\stackrel{}\dashrightarrow$ satisfying $T_{n} \in \mathcal{T}$, $i^{!}K_{n-1} \in (\mathcal{T}_{1})^{\wedge}_{m-1}$ and
$j^{\ast}K_{n-1} \in (\mathcal{T}_{2})^{\wedge}_{n-1}$ for some non-negative integers $m, n$.
  \item[(2)] There exists an $\mathbb{E}_\mathcal{B}$-triangle
$M \stackrel{}{\longrightarrow} F_{n'}\stackrel{}{\longrightarrow} L_{n'-1}\stackrel{}\dashrightarrow$ satisfying $F_{n'} \in \mathcal{F}$, $i^{!}L_{n'-1} \in (\mathcal{F}_{1})^{\vee}_{m'-1}$ and
$j^{\ast}L_{n'-1} \in (\mathcal{F}_{2})^{\vee}_{n'-1}$ for some non-negative integers $m', n'$.
\end{itemize}
\end{lemma}

\begin{proof} We will only prove (1), the proof of (2) is similar.
For any $M \in \mathcal{B}$, we have $j^{\ast}M \in \mathcal{C}$, and there exists a non-negative integer $n$ such that $j^{\ast}M \in (\mathcal{T}_{2})^{\wedge}_{n}$, since $(\mathcal{T}_{2},\mathcal{F}_{2})$ is a bounded hereditary cotorsion pair in $\mathcal{C}$.
Consequently, we obtain an $\mathbb{E}_\mathcal{C}$-triangle
$K^{\mathcal{C}}_{n-1} \stackrel{}{\longrightarrow} T^{\mathcal{C}}_{n} \stackrel{}{\longrightarrow} j^{\ast}M \stackrel{}\dashrightarrow$ with $K^{\mathcal{C}}_{n-1} \in (\mathcal{T}_{2})^{\wedge}_{n-1}$ and $T^{\mathcal{C}}_{n} \in \mathcal{T}_{2}$. Since $i^{!}$ is exact, by Lemma \ref{CY}$(8')$, the functor $j_{\ast}$ is exact.
Applying $j_{\ast}$ to the above $\mathbb{E}_\mathcal{C}$-triangle, we obtain an $\mathbb{E}_\mathcal{B}$-triangle $j_{\ast}K^{\mathcal{C}}_{n-1} \stackrel{}{\longrightarrow} j_{\ast}T^{\mathcal{C}}_{n} \stackrel{}{\longrightarrow} j_{\ast}j^{\ast}M \stackrel{\delta}\dashrightarrow$.
$\eta_{M}$ is the unit of the adjoint pair $(j^{\ast},j_{\ast})$ yields the following commutative diagram:
\begin{equation}\label{11}
\xymatrix{
  j_{\ast}K^{\mathcal{C}}_{n-1}\ar@{=}[d] \ar[r] & H\ar[d]\ar[r] & M \ar[d]_{\eta_{M}} \ar@{-->}[r]^-{\eta_M^*\delta}  &  \\
  j_{\ast}K^{\mathcal{C}}_{n-1} \ar[r]^-{} & j_{\ast}T^{\mathcal{C}}_{n}\ar[r]^-{} & j_{\ast}j^{\ast}M \ar@{-->}[r]^-\delta&  }
\end{equation}
where the first row
\begin{equation}\label{extriangle-M}
j_{\ast}K^{\mathcal{C}}_{n-1}\stackrel{}{\longrightarrow} H\stackrel{}{\longrightarrow} M\stackrel{\eta_M^*\delta}\dashrightarrow
\end{equation}
is an $\mathbb{E}_\mathcal{B}$-triangle.
Applying the exact functor $j^{\ast}$ to (\ref{11}), we get the following commutative diagram:
\begin{equation*}
\xymatrix{
  j^{\ast}j_{\ast}K^{\mathcal{C}}_{n-1}\ar@{=}[d] \ar[r] & j^{\ast}H\ar[d]\ar[r] & j^{\ast}M \ar[d]_{j^{\ast}\eta_M} \ar@{-->}[r]^-{}  &  \\
  j^{\ast}j_{\ast}K^{\mathcal{C}}_{n-1} \ar[r]^-{} & j^{\ast}j_{\ast}T^{\mathcal{C}}_{n}\ar[r]^-{} & j^{\ast}j_{\ast}j^{\ast}M \ar@{-->}[r] &  .}
\end{equation*}
Using Lemma \ref{adjoint}(1) together with Lemma \ref{CY}(1), the morphism $j^{\ast}\eta_M$ is an isomorphism and hence $j^{\ast}H \cong T^{\mathcal{C}}_{n} \in \mathcal{T}_{2}$.

Notice that $i^{\ast}H \in \mathcal{A}$,
thus there exists a non-negative integer $m$ such that $i^{\ast}H \in (\mathcal{T}_{1})^{\wedge}_{m}$,
because $(\mathcal{T}_{1}, \mathcal{F}_{1})$ is a bounded hereditary cotorsion pair in $\mathcal{A}$.
Hence, we have an $\mathbb{E}_\mathcal{A}$-triangle
$K^{\mathcal{A}}_{m-1}\stackrel{}{\longrightarrow}T^{\mathcal{A}}_{m}\stackrel{}{\longrightarrow} i^{\ast}H\stackrel{}\dashrightarrow$ with $K^{\mathcal{A}}_{m-1}\in (\mathcal{T}_{1})^{\wedge}_{m-1}$ and $T^{\mathcal{A}}_{m}\in \mathcal{T}_{1}$.
Since $i_\ast$ is exact, we have that
$i_{\ast}K^{\mathcal{A}}_{m-1}\stackrel{}{\longrightarrow} i_{\ast}T_{m}^{\mathcal{A}} \stackrel{}{\longrightarrow} i_{\ast}i^{\ast}H\stackrel{}\dashrightarrow$ is an $\mathbb{E}_\mathcal{B}$-triangle. For $H\in \mathcal{B}$, by (R5), there exists a commutative diagram:
\begin{equation*}\label{SS}
\xymatrix{
  &i_{\ast}A' \ar[r]&j_{!}j^{\ast}H\ar[rr]^-{\upsilon_H}\ar[dr]_{h_{2}}&  &H\ar[r]^-{h}&i_{\ast}i^{\ast}H &\\
           &                &       &  L \ar[ur]_{h_{1}}& }
\end{equation*}
in $\mathcal{B}$ such that $i_{\ast}A'\stackrel{}{\longrightarrow} j_{!}j^{\ast}H \stackrel{h_{2}}{\longrightarrow} L \stackrel{}\dashrightarrow$
and
$L\stackrel{h_{1}}{\longrightarrow} H \stackrel{h}{\longrightarrow} i_{\ast}i^{\ast}H \stackrel{}\dashrightarrow$ are $\mathbb{E}_\mathcal{B}$-triangles, and $h_{2}$ is compatible. In particular,
$j_{!}j^{\ast}H \stackrel{v_H}{\longrightarrow} H \stackrel{h}{\longrightarrow} i_{\ast}i^{\ast}H$ is right exact.
By \cite[Proposition 3.15]{Na}, we have the following exact commutative diagram:
\begin{equation*}\label{4.10}
\xymatrix{   &   & L \ar[d]_{t}\ar@{=}[r] &  L\ar[d]^{h_{1}} &\\
  & i_{\ast}K^{\mathcal{A}}_{m-1} \ar[r]\ar@{=}[d]  &  T_{n} \ar[d]_{l} \ar[r]&  H \ar[d]^-h \\                                               & i_{\ast}K^{\mathcal{A}}_{m-1} \ar[r]  &  i_{\ast}T^{\mathcal{A}}_{m}  \ar[r]&  i_{\ast}i^{\ast}H & }
\end{equation*}
where the second row
\begin{equation}\label{extriangle-H}
i_{\ast}K^{\mathcal{A}}_{m-1}\stackrel{}{\longrightarrow} T_{n}\stackrel{}{\longrightarrow} H\stackrel{}\dashrightarrow
\end{equation}
is an $\mathbb{E}_\mathcal{B}$-triangle.
Consider the following commutative diagram:
\begin{equation}\label{K}
\xymatrix{
  j_{!}j^{\ast}H\ar[rr]\ar[dr]_-{h_{2}}&  &T_{n}\ar[r]^-{l}&i_{\ast}T^{\mathcal{A}}_{m} &\\
           &           L \ar[ur]_{t}      &       & & }
\end{equation}
where $h_{2}$ is a compatible deflation and $t$ is an inflation. Thus, the first row of (\ref{K}) is right exact. Applying the right exact functor $i^{\ast}$ to (\ref{K}),
using Lemma \ref{CY}(2), we obtain that that $i^{\ast}j_{!}j^{\ast}H=0$.
Thus, by \cite[Lemma 2.10(2)]{WL} together with Lemma \ref{CY}(1), we have that $i^{\ast}T_{n} \cong i^{\ast}i_{\ast} T^{\mathcal{A}}_{m}\cong T^{\mathcal{A}}_{m} \in \mathcal{T}_{1}$.
Applying $j^{\ast}$ to the $\mathbb{E}_\mathcal{B}$-triangle (\ref{extriangle-H}), we get an $\mathbb{E}_\mathcal{C}$-triangle
$j^{\ast}i_{\ast}K^{A}_{m-1} \stackrel{}{\longrightarrow} j^{\ast}T_{n} \stackrel{}{\longrightarrow} j^{\ast}H \stackrel{}\dashrightarrow$. By (R2), we have $j^{\ast}i_{\ast}K^{\mathcal{A}}_{m-1} = 0$, and using \cite[Lemma 2.10(2)]{WL}, we get $j^{\ast}T_{n} \cong j^{\ast}H$. Recall that $j^{\ast}H  \in \mathcal{T}_{2}$. Hence, we have $j^{\ast}T_{n} \cong j^{\ast}H \in \mathcal{T}_{2}$.  Since $i^{\ast}T_{n} \in \mathcal{T}_{1}$ and $j^{\ast}T_{n}\in \mathcal{T}_{2}$, we obtain that $T_{n}\in \mathcal{T}$.

Applying $\rm (ET4)^{op}$ to
$\mathbb{E}_\mathcal{B}$-triangles (\ref{extriangle-M}) and (\ref{extriangle-H}),
we get a  commutative diagram of conflations as follows:
\begin{equation}\label{7}
\xymatrix{
  i_{\ast}K^{\mathcal{A}}_{m-1} \ar@{=}[d] \ar[r] & K_{n-1}  \ar[d] \ar[r] & j_{\ast}K^{\mathcal{C}}_{n-1}  \ar[d] \\
  i_{\ast}K^{\mathcal{A}}_{m-1} \ar[r] & T_{n} \ar[d] \ar[r] &  H \ar[d] \\
  &  M  \ar@{=}[r] & M. }
\end{equation}
Applying $j^*$ to the $\mathbb{E}_\mathcal{B}$-triangle in the first row of (\ref{7}),
by (R2), we have $j^{\ast}i_{\ast}K^{\mathcal{A}}_{m-1} = 0$.
Using \cite[Lemma 2.10(2)]{WL} and Lemma \ref{CY}(1), we obtain that $j^{\ast}K_{n-1} \cong j^{\ast}j_{\ast}K^{\mathcal{C}}_{n-1} \cong K^{\mathcal{C}}_{n-1} \in (\mathcal{T}_{2})^{\wedge}_{n-1}$. Similarly, applying $i^!$ to the $\mathbb{E}_\mathcal{B}$-triangle in the first row of (\ref{7}),
by Lemma \ref{CY}(2), we get $i^{!}j_{\ast}K^{\mathcal{C}}_{n-1} = 0$. Thus,
$i^{!}K_{n-1}\cong i^{!}i_{\ast}K^{\mathcal{A}}_{m-1}\cong K^{\mathcal{A}}_{m-1} \in (\mathcal{T}_{1})^{\wedge}_{m-1}$ holds.
Hence, the $\mathbb{E}_\mathcal{B}$-triangle
$K_{n-1} \stackrel{}{\longrightarrow} T_{n}\stackrel{}{\longrightarrow} M\stackrel{}\dashrightarrow$ satisfies that $T_{n} \in \mathcal{T}$, $i^{!}K_{n-1} \in (\mathcal{T}_{1})^{\wedge}_{m-1}$ and
$j^{\ast}K_{n-1} \in (\mathcal{T}_{2})^{\wedge}_{n-1}$.
\end{proof}

Now we can state our  main result in this section.

\begin{theorem}\label{main-silt} Let $(\mathcal{A}, \mathcal{B}, \mathcal{C}, i^{*}, i_{\ast}, i^{!},j_!, j^\ast, j_\ast)$ be a recollement of extriangulated categories as in {\rm (\ref{recolle})}. Suppose that $\mathcal{B}$ has enough projective objects and $i^{!}, j_{!}$ are exact.

$(1)$ Let $\mathcal{M}_{\mathcal{A}}$ and $\mathcal{M}_{\mathcal{C}}$ be silting subcategories in $\mathcal{A}$ and $\mathcal{C}$ respectively.
If the functor $i^{\ast}$ is exact, then
$$\mathcal{M}_{\mathcal{B}}:=\{B\in \mathcal{B}~|~i^{\ast }B\in\mathcal{M}_{\mathcal{A}}^{\vee},j^{\ast}B\in \mathcal{M}_{\mathcal{C}}^{\vee},i^{!}B\in\mathcal{M}_{\mathcal{A}}^{\wedge},j^{\ast}B\in \mathcal{M}_{\mathcal{C}}^{\wedge}\}$$
is a silting subcategory in $\mathcal{B}$. In addition, if $\mathcal{M}_{\mathcal{A}}=\add M_{A}$ and $\mathcal{M}_{\mathcal{C}} = \add M_{C}$, then $M_{\mathcal{B}} = i_{\ast}M_{A} \oplus j_{!}M_{C}$.

$(2)$ Let $\mathcal{M}$ be a silting subcategory in $\mathcal{B}$.
 \begin{itemize}
   \item[(a)]  If $i_{\ast}i^{!}\mathcal{M}^{\vee} \subseteq \mathcal{M}^{\vee}$ and $i_{\ast}i^{\ast}\mathcal{M}^{\vee} \subseteq \mathcal{M}^{\vee}$,
       then $i^{\ast}\mathcal{M}^{\vee} \cap i^{!}\mathcal{M}^{\wedge}$ is a silting subcategory of $\mathcal{A}$.
       In addition, if $\mathcal{M} = \add M$, then $i^{\ast}M$ is a silting object of $\mathcal{A}$.
   \item[(b)] If $j_{\ast}j^{\ast}\mathcal{M}^{\wedge} \subseteq \mathcal{M}^{\wedge}$ or $j_{!}j^{\ast}\mathcal{M}^{\vee} \subseteq \mathcal{M}^{\vee}$,
       then $j^{\ast}\mathcal{M}^{\vee}\cap j^{\ast}\mathcal{M}^{\wedge}$ is a silting subcategory of $\mathcal{C}$.
       In addition, if $\mathcal{M} = \add M$, then $j^{\ast}M$ is a silting object of $\mathcal{C}$.
 \end{itemize}
\end{theorem}

\begin{proof} (1) Let $(\mathcal{T}_{1},\mathcal{F}_{1})$ and $(\mathcal{T}_{2},\mathcal{F}_{2})$ the corresponding bounded hereditary cotorsion pairs
$(\mathcal{M}_{\mathcal{A}}^{\vee},\mathcal{M}_{\mathcal{A}}^{\wedge})$
and $(\mathcal{M}_{\mathcal{C}}^{\vee},\mathcal{M}_{\mathcal{C}}^{\wedge})$, respectively. Let $(\mathcal{T},\mathcal{F})$ be the glued pair with respect to $(\mathcal{T}_{1},\mathcal{F}_{1})$ and $(\mathcal{T}_{2},\mathcal{F}_{2})$. By \cite[Lemma 4.5]{WL}, it is suffices to check the conditions (CP1) and (HCP) in Definition \ref{hcotors}. We divide the proof into the following steps:

\textbf{Step 1.} Take $T = T_{1} \oplus T_{2} \in \mathcal{T}$. By construction, we have $i^{\ast}T_{1} \oplus i^{\ast}T_{2} \cong i^{\ast}(T_{1} \oplus T_{2}) \in \mathcal{T}_{1}$. Since $\mathcal{T}_{1}$ is closed under direct summands, we have $i^{\ast}T_{1},i^{\ast}T_{2}\in\mathcal{T}_{1}$. Similarly, we have $j^{\ast}T_{1},j^{\ast}T_{2}\in\mathcal{T}_{2}$. Therefore, $\mathcal{T}$ is closed under direct summands. Similarly, we can prove  $\mathcal{F}$ is closed under direct summands. Hence, $(\mathcal{T},\mathcal{F})$ satisfies condition (CP1).

\textbf{Step 2.}
Since $i^{\ast}$ is exact, by \cite[Proposition 3.4(2)]{WL}, there exists an $\mathbb{E}_\mathcal{B}$-triangle $j_! j^\ast T\stackrel{}{\longrightarrow} T\stackrel{}{\longrightarrow} i_\ast i^\ast T\stackrel{}\dashrightarrow$. Applying $\Hom_\mathcal{B}(-,F)$ to above $\mathbb{E}_\mathcal{B}$-triangle,
we get the following long exact sequence
\begin{equation}\label{EXACT11}
\mathbb{E}_{\mathcal{B}}(T,F)\stackrel{}{\rightarrow}
\mathbb{E}_{\mathcal{B}}(j_! j^\ast T,F)\stackrel{}{\rightarrow}
\mathbb{E}^{2}_{\mathcal{B}}(i_\ast i^\ast T,F)\stackrel{}{\rightarrow} \mathbb{E}^{2}_{\mathcal{B}}(T,F)\stackrel{}{\rightarrow}
\mathbb{E}^{2}_{\mathcal{B}}(j_! j^\ast T,F)\stackrel{}{\rightarrow} \stackrel{}\cdots
\end{equation}

By Lemma \ref{prop 311}(2), we obtain that $\mathbb{E}^{n}_\mathcal{B}(i_{\ast}i^{\ast}T,F)\cong\mathbb{E}^{n}_\mathcal{A}(i^{\ast}T,i^{!}F)=0$, since $i^{\ast}T\in\mathcal{T}_1$ and $i^{!}F\in\mathcal{F}_1$. Similarly, by Lemma \ref{prop 311}(3), we have $\mathbb{E}^{n}_\mathcal{B}(j_{!}j^{\ast}T,F)\cong\mathbb{E}^{n}_\mathcal{C}(j^{\ast}T,j^{\ast}F)=0$. By the exactness of the long exact sequence (\ref{EXACT11}), we conclude that  $\mathbb{E}^{n}_\mathcal{B}(T,F)=0$. Thus, $(\mathcal{T},\mathcal{F})$ satisfes condition (HCP).

\textbf{Step 3.} We still need to show that $(\mathcal{T},\mathcal{F})$ is bounded.
For any $M \in \mathcal{B}$,
by Lemma \ref{main-silt-lemma}, we have an $\mathbb{E}_\mathcal{B}$-triangle
$K_{n-1} \stackrel{}{\longrightarrow} T_{n}\stackrel{}{\longrightarrow} M\stackrel{}\dashrightarrow$ with $T_{n} \in \mathcal{T}$, $i^{!}K_{n-1} \in (\mathcal{T}_{1})^{\wedge}_{m-1}$ and
$j^{\ast}K_{n-1} \in (\mathcal{T}_{2})^{\wedge}_{n-1}$.
Repeatedly, we have an $\mathbb{E}_\mathcal{B}$-triangle
$K_{i-1} \stackrel{}{\longrightarrow} T_{i}\stackrel{}{\longrightarrow} K_{i}\stackrel{}\dashrightarrow$ with $T_{i} \in \mathcal{T}$ for each $i \leq n-1$.
This implies that $M \in (\mathcal{T})_{n}^{\wedge}$. Consequently, we have $\mathcal{B} \subseteq \mathcal{T}^{\wedge}$. It is obvious that $\mathcal{T}^{\wedge} \subseteq \mathcal{B}$. Therefore, we conclude that $\mathcal{B} = \mathcal{T}^{\wedge}$. Similarly, we have $\mathcal{B} = \mathcal{F}^{\vee}$. Thus, $(\mathcal{T}, \mathcal{F})$ is a bounded hereditary cotorsion pair in $\mathcal{B}$. In particular, $\mathcal{T} \cap \mathcal{F}$ is a silting subcategory of $\mathcal{B}$.

(2) Let $(\mathcal{U}, \mathcal{V})$ denote the associated bounded hereditary cotorsion pair
$(\mathcal{M}^{\vee}, \mathcal{M}^{\wedge})$ in $\mathcal{B}$.
We first show that $i^{\ast}\mathcal{U}$ is closed under direct summands.
Let $A = A_{1} \oplus A_{2} \in i^{\ast}\mathcal{U}$.
It is obvious that $i_{\ast}A_{1} \oplus i_{\ast}A_{2} \cong i_{\ast}(A_{1} \oplus A_{2}) = i_{\ast}A \in \mathcal{B}$.
Since $i_{\ast}A \in i_{\ast}i^{\ast}\mathcal{U} \subseteq \mathcal{U}$ and $\mathcal{U}$ is closed under direct summands, it follows that both $i_{\ast}A_{1}$ and $i_{\ast}A_{2}$ belong to $ \mathcal{U}$. This implies that $A_{1} \cong i^{\ast}i_{\ast}A_{1} \in i^{\ast}\mathcal{U}$ and $A_{2} \cong i^{\ast}i_{\ast}A_{2} \in i^{\ast}\mathcal{U}$. Therefore, $i^{\ast}\mathcal{U}$ is closed under direct summands. The proof that $i^{!}{V}$ is closed under direct summands follows a similar argument. Thus, $(i^{\ast}{U},i^{!}{V})$ satisfies (CP1) in Definition \ref{hcotors}. By \cite[Lemma 4.5]{WL}, $(i^{\ast}{U},i^{!}{V})$ satisfy (CP2), (CP3) and (CP4) in Definition \ref{hcotors}.
By \cite[Theorem 4.6]{MZ}, $(i^{\ast}{U},i^{!}{V})$ satisfies (HCP) in Definition \ref{hcotors}.
It means that $(i^{\ast}\mathcal{U}, i^{!}\mathcal{V})$ is a hereditary cotorsion pair in $\mathcal{A}$.


It remains to show that $(i^{\ast}{U},i^{!}{V})$ is bounded.
For any $X \in \mathcal{A}$, then $i_{\ast}X \in \mathcal{B}$
and there exists a non-negative integer $n$ such that $i_{\ast}X \in (\mathcal{U}_{n})^{\wedge}$, since $(\mathcal{U},\mathcal{V})$ is a bounded hereditary cotorsion pair in $\mathcal{B}$.
Hence, we obtain an $\mathbb{E}_\mathcal{B}$-triangle
$K_{n-1} \stackrel{}{\longrightarrow} U_{n}\stackrel{}{\longrightarrow} i_{\ast}X \stackrel{}\dashrightarrow$ with $U_{n} \in \mathcal{U}$ and $K_{n-1} \in (\mathcal{U}_{n-1})^{\wedge}$.
We also have an $\mathbb{E}_\mathcal{B}$-triangle
$K_{i-1} \stackrel{}{\longrightarrow} U_{i}\stackrel{}{\longrightarrow} K_{i}\stackrel{}\dashrightarrow$ with $U_{i} \in \mathcal{U}$
for each $i \leq n-1$.
Applying the exact functor $i^{!}$ to above $\mathbb{E}_\mathcal{B}$-triangles, we obtain the following $\mathbb{E}_\mathcal{A}$-triangles
$i^{!}K_{n-1} \stackrel{}{\longrightarrow} i^{!}U_{n}\stackrel{}{\longrightarrow} i^{!}i_{\ast}X \stackrel{}\dashrightarrow$ with $i^{!}i_{\ast}X \cong X \in \mathcal{A}$, and
$i^{!}K_{i-1} \stackrel{}{\longrightarrow} i^{!}U_{i}\stackrel{}{\longrightarrow} i^{!}K_{i}\stackrel{}\dashrightarrow$.
Since $i_{\ast}i^{!}\mathcal{U} \subseteq \mathcal{U}$, we have $i^{!}U_{i} \cong i^{\ast}i_{\ast}i^{!}U_{i} \in i^{\ast}\mathcal{U}$ for each $i \leq n$.
This implies that $X \in (i^{\ast}\mathcal{U})_{n}^{\wedge}$.
Consequently, we have $\mathcal{C} \subseteq (i^{\ast}\mathcal{U})^{\wedge}$. It is obvious that $(i^{\ast}\mathcal{U})^{\wedge} \subseteq \mathcal{C}$. Therefore, $\mathcal{C} = (i^{\ast}\mathcal{U})^{\wedge}$ holds. Similarly, we have $\mathcal{C} = (i^{!}\mathcal{V})^{\vee}$.
Thus, $(i^{\ast}{\mathcal{U}},i^{!}{\mathcal{V}})$ is a bounded hereditary cotorsion pair in $\mathcal{C}$. In particular, $i^{\ast}\mathcal{U} \cap i^{!}\mathcal{V}$ is a silting subcategory of $\mathcal{A}$. We complete the proof of (a). The proof of (b) is similar.
\end{proof}

As an immediate corollary,   we recover the corresponding results in the triangulated category case, see \cite[Theorem 8.2.3]{Bo} and \cite[Theorem 2.2(3),(4)]{LVY}.
\begin{corollary}
Let $(\mathcal{A}, \mathcal{B}, \mathcal{C}, i^{*}, i_{\ast}, i^{!},j_!, j^\ast, j_\ast)$ be a recollement of triangulated categories as in {\rm (\ref{recolle})}.

$(1)$ Let $\mathcal{M}_{\mathcal{A}}$ and $\mathcal{M}_{\mathcal{C}}$ be silting subcategories in $\mathcal{A}$ and $\mathcal{C}$, respectively.
Then
$$\mathcal{M}_{\mathcal{B}}:=\{B\in \mathcal{B}~|~i^{\ast }B\in\mathcal{M}_{\mathcal{A}}^{\vee},j^{\ast}B\in \mathcal{M}_{\mathcal{C}}^{\vee},i^{!}B\in\mathcal{M}_{\mathcal{A}}^{\wedge},j^{\ast}B\in \mathcal{M}_{\mathcal{C}}^{\wedge}\}$$
is a silting subcategory in $\mathcal{B}$. In addition, if $\mathcal{M}_{\mathcal{A}}=\add M_{A}$ and $\mathcal{M}_{\mathcal{C}} = \add M_{C}$, then $M_{\mathcal{B}} = i_{\ast}M_{A} \oplus j_{!}M_{C}$.

$(2)$ Let $\mathcal{M}$ be a silting subcategory in $\mathcal{B}$.
 \begin{itemize}
   \item[(a)]  If $i_{\ast}i^{!}\mathcal{M}^{\vee} \subseteq \mathcal{M}^{\vee}$ and $i_{\ast}i^{\ast}\mathcal{M}^{\vee} \subseteq \mathcal{M}^{\vee}$,
       then $i^{\ast}\mathcal{M}^{\vee} \cap i^{!}\mathcal{M}^{\wedge}$ is a silting subcategory of $\mathcal{A}$.
       In addition, if $\mathcal{M} = \add M$, then $i^{\ast}M$ is a silting object of $\mathcal{A}$.
   \item[(b)] If $j_{\ast}j^{\ast}\mathcal{M}^{\wedge} \subseteq \mathcal{M}^{\wedge}$ or $j_{!}j^{\ast}\mathcal{M}^{\vee} \subseteq \mathcal{M}^{\vee}$,
       then $j^{\ast}\mathcal{M}^{\vee}\cap j^{\ast}\mathcal{M}^{\wedge}$ is a silting subcategory of $\mathcal{C}$.
       In addition, if $\mathcal{M} = \add M$, then $j^{\ast}M$ is a silting object of $\mathcal{C}$.
 \end{itemize}
\end{corollary}

\begin{proof} By \cite[Example 3.4(1)]{AT} and \cite[Remark 4.2]{AT}, bounded co-t-structures coincides with the definition of bounded hereditary cotorsion pairs of triangulated categories.
\end{proof}



\end{document}